\documentclass[a4paper]{amsart}

\usepackage{amsfonts}
\usepackage{amsmath}
\usepackage{amssymb}
\usepackage{amsthm,color}
\usepackage{enumerate}
\usepackage{mathtools}
\mathtoolsset{showonlyrefs}
 \usepackage{transparent}
 \usepackage{stackengine,tikz}
 \usepackage{wrapfig}
\usepackage{graphicx}

\usepackage{color,graphicx,enumerate,wrapfig,amssymb,mathtools}
\usepackage{epsfig}
\usepackage{pxfonts}
\usepackage{graphicx}

\usepackage{eucal}

\newcommand{\R}{{\mathbb R}}

\newcommand{\N}{{\mathbb N}}

\newcommand{\cM}{{\mathcal M}}

\newcommand{\cK}{{\mathcal K}}
\newcommand{\cA}{{\mathcal A}}

\newcommand{\e}{\varepsilon}
\newcommand{\al}{\alpha}
\newcommand{\be}{\beta}
\newcommand{\de}{\delta}

\newcommand{\la}{\lambda}

\newcommand{\hg}{\hat{g}}
\newcommand{\hG}{\hat{G}}

\newcommand{\dist}{\operatorname{dist}}

\newcommand{\loc}{\operatorname{loc}}
\newcommand{\ddiv}{\operatorname{div}}
\newcommand{\D}{\nabla}
\newcommand{\p}{\partial}

\newcommand{\mean}[1]{\langle{#1}\rangle}

\usepackage{amsmath}
\usepackage{amsfonts}
\usepackage{caption}
\usepackage{subcaption}
\usepackage{esint}

\newtheorem{theorem}{Theorem}
\theoremstyle{plain}

\newtheorem{lemma}{Lemma}
\newtheorem{remark}{Remark}
\newtheorem{proposition}{Proposition}

\numberwithin{equation}{section}

\makeatletter
\@namedef{subjclassname@2020}{%
  \textup{2020} Mathematics Subject Classification}
\makeatother

\begin{document}

\author{Seongmin Jeon}
\address[Seongmin Jeon]
{Department of Mathematics \newline 
\indent KTH Royal Institute of Technology \newline 
\indent 100 44 Stockholm, Sweden} 
\email[Seongmin Jeon]{seongmin@kth.se}

\author{Henrik Shahgholian}
\address[Henrik Shahgholian]
{Department of Mathematics \newline 
\indent KTH Royal Institute of Technology \newline 
\indent 100 44 Stockholm, Sweden} 
\email[Henrik Shahgholian]{henriksh@kth.se}

\title[Convexity for nonlinear  elliptic  free boundaries]
{Convexity for    free boundaries with  singular term \\  (nonlinear elliptic case)}

\date{\today}
\thanks{HS is supported by Swedish Research Council. This project was carried out  while the authors stayed at Institute Mittag Leffler (Sweden), during the program Geometric aspects of nonlinear PDE}

\begin{abstract}
    We consider a free boundary  problem in an exterior domain \begin{align*}
        \begin{cases}
        Lu=g(u)&\text{in }\Omega\setminus K,\\
        u=1 & \text{on }\p K,\\
        |\D u|=0 &\text{on }\p \Omega,
        \end{cases}
    \end{align*} 
        where $K$ is a (given)  convex and compact set in $\R^n$ ($n\ge2$), $\Omega=\{u>0\}\supset K$ is an unknown set, and $L$ is  either  a fully nonlinear  or the $p$-Laplace  operator.
        Under  suitable assumptions on  $K$ and  $g$, we prove the existence of a nonnegative quasi-concave solution to the above problem.
    
    We also consider the cases when the set $K$ is contained in $\{x_n=0\}$, and obtain  similar results.

\end{abstract}

\maketitle 

\tableofcontents

\section{Introduction}
\subsection{Background}
Let $K$ be a (given) compact convex set in $\R^n$, $n\ge 2$, and $L$ be a nonlinear elliptic differential operator (specified below).
 For a given function $g$, we consider the following obstacle type free boundary problem
\begin{align}
    \label{eq:sol}
    \begin{cases}
    Lu=g(u)&\text{in }\Omega\setminus K,\\
    u=1 & \text{on }\p K,\\
    |\D u|=0 &\text{on }\p \Omega.
    \end{cases}
\end{align}
The assumptions on the right-hand side (r.h.s.) $g$ are specified  in \eqref{eq:assump-rhs} below.
 It is noteworthy that $g$ can be discontinuous and highly singular near $0$, e.g., $g(u)\approx u^a$ with $-1<a<0$ when $0<u\approx 0$.

We shall  consider two kinds of nonlinear operators $L$: 
\begin{itemize}
\item  The fully nonlinear operator $F(D^2u)$ (see below for a defintion).
\item  The $p$-laplace operator $\Delta_pu=\ddiv\left(|\D u|^{p-2}\D u\right)$, $1<p<\infty$.
\end{itemize}

Free boundary problems with highly singular r.h.s. like $u^a$, $-1<a<0$, were studied by Alt-Phillips \cite{AltPhi86} for the Laplacian case. The problems with fully nonlinear operator were treated by Ara\'{u}jo-Teixeira \cite{AraTei13}, and the $p$-Laplacian (with $2\le p<\infty$) by Leit\~{a}o-de Queiroz-Teixeira \cite{LeideQTei15}.

The main objective of this paper is to prove the existence of a quasi-concave solution for \eqref{eq:sol}. Note that a function is called quasi-concave if it has convex super-level sets. 

Convexity configurations in elliptic and  parabolic PDEs have been extensively studied in the literature; e.g.    \cite{BiaLonSal09},  
\cite{ColSal03},
 \cite{CuoSal06}, 
\cite{ElH21},
\cite{ElHSha20}, 
 \cite{GreKaw09}, 
\cite{Kor90}, \cite{LinPri07}, \cite{Pet01}.
 In particular, the second author and El Hajj \cite{ElHSha20} recently investigated the convexity problem concerning the Laplace operator (with both obstacle-type and Bernoulli-type\footnote{Bernoulli-type free 
  boundary problems refers to the case when $|\nabla u| = h(x)$, with given $h >0$, and usually $h$ is constant, but can also have some concavity property.}
  boundary conditions). The present paper generalizes some of its result for Laplacian to nonlinear operators.

We also consider the  case  where  the compact convex set $K$ is contained in $\{x_n=0\}$ and the solution $u$ is defined in $\R^n_+:=\{x=(x',x_n)\in\R^n:x_n>0\}$. In this case the problem is defined as 
\begin{align}
    \label{eq:sol-thin}
    \begin{cases}
     Lu=g(u)&\text{in }\Omega,\\
    u=1 & \text{on } K^{\mathrm{o}},\\
    u=0&\text{on }\p \Omega\setminus K,\\
    |\D u|=0 &\text{on }(\p \Omega\setminus K)
    \cap \R^n_+,
    \end{cases}
\end{align}
where  $\Omega:=\{u>0\}\subset \R^n_+$,
and  $K^{\mathrm{o}}$ is understood as the interior of the set $K$ relative to $\{x_n=0\}$. As above, we aim to find a quasi-concave solution of \eqref{eq:sol-thin}. A similar convexity problem was treated by Lindgren-Privat \cite{LinPri07} for the Bernoulli-type free boundary problem regarding the Laplace operator.

One last problem we deal with is  the following: for a compact convex set $K\subset\{x_n=0\}$, we look for a nonnegative function $u:\R^n\to\R$ with $\Omega:=\{u>0\}\supset K$ such that 
\begin{align}
    \label{eq:sol-hyb}
    \begin{cases}
    Lu=g(u)&\text{in }\Omega,\\
    u=1&\text{on }K^{\mathrm{o}},\\
    |\D u|=0&\text{on }\p \Omega.
    \end{cases}
\end{align}

\subsection{Approach and methodology}
Our approach is based on quasi-concave rearrangements. In doing so, we first consider the 
regularized problem, and  prove the quasi-concavity of the regular  solution, which for definiteness we denote by $v$. The  so called quasi-concave envelope $v^*$ of $v$  is defined as the smallest quasi-concave function greater than or equal to $v$. Equivalently, the super-level sets of $v^*$ are the closed convex hulls of the corresponding super-level sets of $v$. As $v^*\ge v$ by definition, it is sufficient to prove $v\ge v^*$ to obtain the quasi-concavity of $v$.

This method was suggested by Kawohl \cite{Kaw98} and was first employed by Colesanti-Salani \cite{ColSal03}. This technique was exploited in many problems, such as \cite{BiaLonSal09}, \cite{CuoSal06}, \cite{ElHSha20}, etc. We also refer those papers for properties of the quasi-concave envelopes.

In our framework, the highly singular r.h.s. makes the study of the existence of quasi-concave solutions to \eqref{eq:sol}-\eqref{eq:sol-hyb} quite delicate, and substantially technical.
 To circumvent this difficulty, we approximate the r.h.s. $g$ by more regular functions. For this purpose, we first consider regularized problems by replacing $g$ by $h$ (see \eqref{eq:assump-h} for the condition on $h$ as well as the definition of a function $H$ in \eqref{eq:H}).


\subsection{Main results}

To state our main results, we identify $\R^{n-1}$ with $\{x_n=0\}=\R^{n-1}\times\{0\}\subset \R^n$, and consider the following classes of sets \begin{align*}
    &\cA:=\{K\subset \R^n\,:\, \text{$K$ is a compact convex set with a nonempty interior}\},\\
    &\tilde{\cA}:=\{K\subset\R^{n-1}\,:\, \text{$K$ is compact convex with nonempty interior relative to $\R^{n-1}$}\}.
\end{align*}

For the r.h.s. $g:\R\to\R$, we fix $-1<a<0$, $0<b<1$ and $C_1>c_1>0$, and assume \begin{align}\label{eq:assump-rhs}
    \begin{cases}
    \text{- }g=0\text{ on }(-\infty,0],\\
    \text{- }g\text{ is nonnegative and continuous on }(0,\infty),\\
    \text{- }c_1t^b\le g(t)\le C_1t^a\text{ for }0< t<1,\\
    \text{- }g\text{ is bounded on }[1,\infty).
    \end{cases}
\end{align}
In \eqref{eq:assump-rhs}, the range $-1<a<0$ can be relaxed to $-1< a<1$. This follows from the simple observation that if $g$ satisfies \eqref{eq:assump-rhs} with nonnegative $a\in[0,1)$, then it  does so with negative $a\in(-1,0)$. The lower bounded $a>-1$ is asked for the $C^{1,\al}$-regularity of solutions to the above problems, see Appendix~\ref{appen:grad-holder}.

To specify the fully nonlinear operator   $F(D^2u)$, let $S=S(n)$ be the space of $n\times n$ symmetric matrices. For constants $\Lambda\ge \lambda>0$, we let $\cM^+_{\la,\Lambda}$, $\cM^-_{\la,\Lambda}$ be the extremal Pucci operators 
$$
\cM^+_{\la,\Lambda}(M)=\Lambda\sum_{e_i>0}e_i+\la\sum_{e_i<0}e_i,\qquad \cM^-_{\la,\Lambda}(M)=\la\sum_{e_i>0}e_i+\Lambda\sum_{e_i<0}e_i,
$$
where $e_i$'s are eigenvalues of $M\in S$. We assume $F:S\to \R$ satisfies
\begin{align}
    \label{eq:assump-fully-nonlinear}
    \begin{cases}
    \text{- }F\text{ is uniformly elliptic, i.e., there are constants $\Lambda\ge\la>0$ such that}\\
    \qquad \cM^-_{\la,\Lambda}(M-N)\le F(M)-F(N)\le \cM^+_{\la,\Lambda}(M-N)\text{ for every } M,N\in S,\\
    \text{- }F(0)=0,\\
    \text{- $F$ is concave},\\
    \text{- $F$ is homogeneous of degree $1$: i.e., }    \  F(rM)=rF(M) \ \forall  \ r>0, \ M\in S.
    \end{cases}
\end{align}

The main results in this paper are the following:

\begin{theorem}\label{thm:quasi-concave}
Let $K\in\cA$ and $L$ be either $p$-Laplacian or the fully nonlinear operator $F$ satisfying \eqref{eq:assump-fully-nonlinear}. Suppose  $g:\R\to \R$ satisfies \eqref{eq:assump-rhs}, and when $L$ is $p$-Laplacian  assume $0<b<\min\{1,p-1\}$. Then there exists a nonnegative and quasi-concave function $u$ with bounded $\Omega=\{u>0\}$ solving \eqref{eq:sol}.
\end{theorem}

Since the level sets of   quasi-concave functions are  convex, the theorem implies (after extending $u=1$ in $K$) that the super-level set $\{u\ge l\}$ is convex for every $l>0$.

\begin{theorem}\label{thm:quasi-concave-thin}
Suppose $K\in\tilde{\cA}$ and let $g$ and $L$ be as in Theorem~\ref{thm:quasi-concave}. Then there is a nonnegative quasi-concave solution $u$ to \eqref{eq:sol-thin} with bounded $\Omega=\{u>0\}$.
\end{theorem}

We remark that 
 solutions to equation  \eqref{eq:sol-thin} (to be constructed in the proof of Theorem \ref{thm:quasi-concave-thin}) are not continuous on $\partial K$. This, however, does not affect the convexity of the super-level sets for the solution in Theorem \ref{thm:quasi-concave-thin}.

\begin{theorem}
\label{thm:quasi-concave-hyb}
Let $K\in\tilde{\cA}$ and $g$, $L$ be as above. Then there exists a nonnegative quasi-concave solution $u$ to \eqref{eq:sol-hyb} with compact support.
\end{theorem}

When the operator $L$ is $p$-laplace operator, we actually prove in Theorems~\ref{thm:quasi-concave}-\ref{thm:quasi-concave-hyb} the existence of energy minimizers of the corresponding functionals, that  become the solutions of \eqref{eq:sol}-\eqref{eq:sol-hyb}.

\begin{remark}\label{rem:unique}(Uniqueness) 
The uniqueness for solutions to the free boundary problem studied in this paper in general, and without any geometric/convexity condition, fails.
For the convex regime (treated here) the uniqueness may well be true, but requires far advanced  technical apparatus than used in this paper. Indeed, one can easily see that uniqueness in the class of $C^{1,Dini}$ domains is true, using the Lavrentiev Principle used in the proofs below. 
For example, for $D$, $D_{t_0}$ and $D^*$ as in \emph{Step 2} in Lemma~\ref{lem:quasi-concave-const}, 
we cannot exclude the possibility of  a point $x^0\in D^*\cap D_{t_0}$ such that $D^*$ has a singularity at $x^0$, see Fig. \ref{fig:singularity} (in the figure, the convex envelope $D^*$ of $D$ is the triangle containing $D$).

It is however plausible that one can with some further study of the regularity of the free boundary, as well as that of the solutions, obtain a general uniqueness theory in the  convex regime, or even in the starshaped regime. 

As these technical issues are outside the scope of this paper, we only conjecture what we believe this to be true.

\smallskip

\noindent
{\bf Conjecture 1:}  The solution to our free boundary problem, in the convex setting, is unique.

\smallskip

It is noteworthy that the corresponding interior problem in general do not admit a unique solution, even in the case when $K$ is a ball. The existence of solution with convex level sets, is an open problem.

\noindent
{\bf Conjecture 2:} For   the interior problem, i.e. when $K \supset \Omega $, 
there  exists  of a solution with  convex sub-level sets, i.e. convexity of $\{ u < l \}$. It is probable  that all solutions in the interior case have this convexity property.

\end{remark}

\begin{figure}[h]
\begin{picture}(100,100)(0,0)
\put(0,0){\includegraphics[height=100pt]{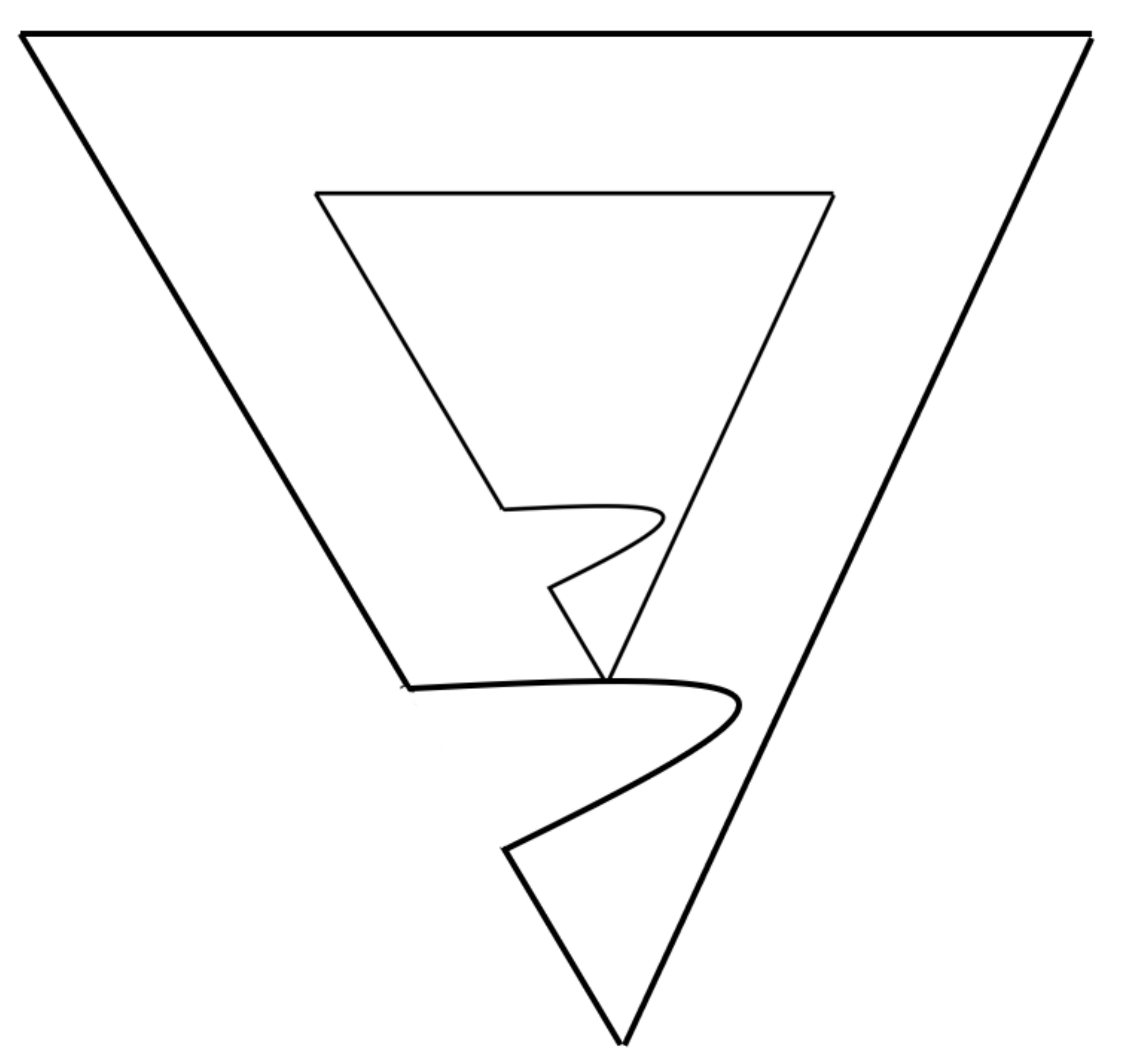}}
\put(75,65){\small $D$}
\put(80,40){\small $D_{t_0}$}
\put(55,29){\scriptsize $x^0$}
\end{picture}
\caption{ Related to Remark~\ref{rem:unique}}
\label{fig:singularity}
\end{figure}
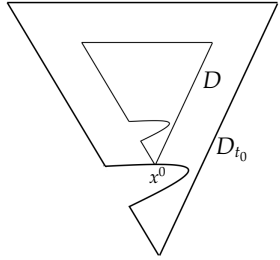


\subsection{Notation}
We denote the points of $\R^n$ by $x=(x',x_n)$, where $x'=(x_1,\cdots,x_{n-1})\in\R^{n-1}$, and identify $\R^{n-1}$ with $\R^{n-1}\times\{0\}$. By $\R^n_{\pm}$, we mean open half spaces $\{(x',x_n)\in\R^n:\pm x_n>0\}$.

We denote balls of radius $r$ by 
\begin{align*}
    &B_r(x^0):=\{x\in\R^n:|x-x^0|<r\}:\quad\text{ball in }\R^n,\\
    &B_r^\pm(x^0):=B_r(x^0)\cap\{\pm x_n>0\}:\quad\text{half ball in }\R^n,\\
    &B'_r(x^0):=B_r(x^0)\cap\{x_n=0\}:\quad\text{ball in }\R^{n-1}.
\end{align*}

Given a function $u:\R^n\to\R$, we will write $$
\p_{x_i}u=\frac{\p u}{\p x_i}, \quad \p_{x_ix_j}u=\frac{\p^2u}{\p x_ix_j},\quad i,j=1,\cdots,n.
$$
We denote the gradient of $u$ by $$
\D u=D u=(\p_{x_i}u,\cdots,\p_{x_n}u).
$$
We also indicate by $D^2u$ the Hessian of $u$, i.e., the $n\times n$ matrix with entries $\p_{x_ix_j}u$.

For a set $A\subset \R^n$, we denote the distance function from $A$ by
$$
d_A(x):=\dist(x,A).
$$

For $-1<a<0$, we fix the following constant throughout this paper 
\begin{equation}\label{beta}
\be:=\frac2{1-a}\in(1,2).
\end{equation}


\section{Proof of Theorem~\ref{thm:quasi-concave}}

In  this section we  prove Theorem~\ref{thm:quasi-concave}. 
We  treat   the fully nonlinear case in Subsection~\ref{subsec:fully-nonlinear} and the $p$-Laplacian case in Subsection~\ref{subsec:p-lap}. 
As mentioned above, in each case we first consider the regularized problem. For this purpose, we define two functions $h$ and $H$ as follows: let $h:\R\to\R$ be a function satisfying for some constant $\e_1>0$ 
\begin{align}\label{eq:assump-h}
    \begin{cases}
    \text{- }h=0\text{ on }(-\infty,0],\\
    \text{- }h\text{ is bounded, Lipschitz continuous and strictly positive on }(0,\infty),\\
    \text{- }h(t)\ge c_1t^b \text{ for }0<t<1 \text{ and } h(t)=1\text{ for }0<t<\e_1,\\
    \text{- }h\le g+2\text{ on }(0,\infty),
    \end{cases}
\end{align}
and define  $H:\R\to\R$ by
\begin{align}
    \label{eq:H}
    H(t):=\int_{-\infty}^th(s)\,ds.
\end{align}
Notice that $H$ is  nonnegative and nondecreasing  in $\R$, and $H=0$ on $(-\infty,0)$.

\subsection{Fully nonlinear case}\label{subsec:fully-nonlinear}
Here we  deal with  the fully nonlinear operator $F(D^2u)$, by first showing existence with the regularized r.h.s.

\begin{lemma}\label{lem:quasi-concave-const}
Let $K\in \cA$ and $h$ be a function satisfying \eqref{eq:assump-h}. Then there exists a solution $v$ of the problem
\begin{align}
    \label{eq:sol-const}
    \begin{cases}
    F(D^2v)=h(v)&\text{in }\R^n\setminus  K,\\
    v=1&\text{on }\p K,
\end{cases}\end{align} 
which is compactly-supported, nonnegative and quasi-concave (after extending $v=1$ on $K$).
\end{lemma}

\begin{proof}
\emph{Step 1.} 
For each $j\in \N$, large enough so that $1/j<\e_1$, we define a function $h^j:\R\to \R$ by
\begin{align}\label{eq:h^j}
    h^j(t):=\begin{cases}
    0,&t\le 1/(2j),\\
    2jt-1,&1/(2j)< t< 1/j,\\
    h(t),&t\ge 1/j.
    \end{cases}
\end{align}  
By the definition of $h$, and that $1/j<\e_1$,  each $h^j$ is Lipschitz in $\R$ and $h^j\nearrow h$ as $j\to\infty$.
 
 We claim that for some large constant $R_0>1$, independent of $j$, with $B_{R_0}\Supset K$ 
  there exists a nonnegative solution $v^j:B_{2R_0}\setminus K\to\R$ to \begin{align}
    \label{eq:sol-const-reg}
    \begin{cases}
    F(D^2v^j)=h^j(v^j)&\text{in }B_{2R_0}\setminus K,\\
    v^j=1&\text{on }\p K,\\
    v^j=\frac1{2j}&\text{on }\p B_{2R_0},
    \end{cases}
\end{align}
with $v^j=\frac1{2j}$ in $B_{2R_0}\setminus\overline{B_{R_0}}$.

To find a solution to \eqref{eq:sol-const-reg}, we will follow the idea in the proof of Theorem~2.1 in \cite{RicTei11}. We first construct  a subsolution to  \eqref{eq:sol-const-reg}. Fix small constants $\rho>0$ and $\tau>0$ depending only on $a$, $\la$, $\Lambda$, $K$, to be specified below, and define a subset of $K$ 
$$
K^\tau:=\{x\in K\,:\, d_{\p K}(x)\ge\tau\} \Subset K.
$$
Note that $K^\tau$ is not empty if $\tau>0$ is small enough, 
 and that $d_{K^\tau}(x)\ge\tau$ 
for $x\in\R^n\setminus K$. 
Recall  the notation \eqref{beta},    $\be=\frac2{1-a}$,  and define 
\begin{align}
    \label{eq:subsol}
    w_\sharp(x):=\begin{cases}
    \frac1{\rho^\be}\left[\tau+\rho-d_{K^\tau}(x)\right]^\be,& \tau<d_{K^\tau}(x)<\tau+\rho,\\
    0,& \qquad d_{K^\tau}(x)\ge\tau+\rho.
    \end{cases}
\end{align}
Then $w_\sharp=1$ on $\p K$, and  for  $\{x\in\R^n\setminus K\,:\,\tau<d_{K^\tau}<\tau+\rho\}$ we can compute
\begin{align*}
D^2w_\sharp=\frac{\be(\be-1)}{\rho^\be}\left(\tau+\rho-d_{K^\tau}\right)^{\be-2}\D d_{K^\tau}\otimes\D d_{K^\tau}-\frac{\be}{\rho^\be}\left(\tau+\rho-d_{K^\tau}\right)^{\be-1}D^2d_{K^\tau},
\end{align*}
and obtain 
\begin{align*}
    F(D^2w_\sharp)\ge& \cM^-_{\la,\Lambda}(D^2w_\sharp)\\
    \ge& \frac{\be(\be-1)}{\rho^\be}\left(\tau+\rho-d_{K^\tau}\right)^{\be-2}\cM^-_{\la,\Lambda}\left(\D d_{K^\tau}\otimes\D d_{K^\tau}\right) \\    
    &    -\frac\be{\rho^\be}\left(\tau+\rho-d_{K^\tau}\right)^{\be-1}\cM^+_{\la,\Lambda}\left(D^2d_{K^\tau}\right).
\end{align*}
Since the eigenvalues of $\D d_{K^\tau}\otimes\D d_{K^\tau}$ are $|\D d_{K^\tau}|^2,0,\cdots,0$ and $|\D d_{K^\tau}|\ge1$, we have $\cM^-_{\la,\Lambda}(\D d_{K^\tau}\otimes\D d_{K^\tau})\ge \la$. Moreover, $\cM^+_{\la,\Lambda}(D^2 d_{K^\tau})\le C(\la,\Lambda,K,\tau)$ in $\{\tau<d_{K^\tau}<\tau+\rho\}$ (see e.g. Theorem~4.8 in \cite{Fed59}), and thus we have for small $\rho>0$ $$
F(D^2w_\sharp)\ge\frac{c(a,\la,\Lambda,\tau)}{\rho^\be}(\tau+\rho-d_{K^\tau})^{\be-2}\quad \text{in } \{\tau<d_{K^\tau}<\tau+\rho\}.
$$ 
 To prove $F(D^2w_\sharp)\ge h^j(w_\sharp)$ in $\{\tau<d_{K^\tau}<\tau+\rho\}$,
  we observe that in the set 
  $\{\tau<d_{K^\tau}<\tau+\rho\}$ $$
F(D^2w_\sharp)\ge\frac{c}{\rho^\be}(\tau+\rho-d_{K^\tau})^{\be-2}\ge\frac{c}{\rho^2}
$$
and 
$$
w_\sharp^a=\frac{(\tau+\rho-d_{K^\tau})^{\be a}}{\rho^{\be a}}=\frac{\rho^{\be(1-a)}}{c}\cdot\frac{c(\tau+\rho-d_{K^\tau})^{\be-2}}{\rho^\be}\le\frac{\rho^{\be(1-a)}}cF(D^2w_\sharp).
$$
By using the above two estimates for $F(D^2w_\sharp)$ from below, and that 
 $\frac{c}{\rho^2}$ and $\frac{c}{\rho^{\be(1-a)}}$ are large when $\rho$ is small, we obtain 
 \begin{align*}
    F(D^2w_\sharp)\ge\frac12\left(\frac{c}{\rho^{\be(1-a)}}w_\sharp^a+\frac{c}{\rho^2}\right)&\ge\frac12(2g(w_\sharp)+4)=g(w_\sharp)+2\ge h^j(w_\sharp)
\end{align*}
for $\rho>0$ small enough. Here, we used \eqref{eq:assump-rhs} in the second step and the inequalities $h^j\le h\le g+2$ in the last step.

 We now define a continuous function 
\begin{align*}
    w_{\sharp,j}(x):=
    \begin{cases}
    w_\sharp(x),&\tau<d_{K^\tau}(x)<\tau+\rho-\rho\left(\frac1{2j}\right)^{1/\be}, \medskip \\ 
    \frac1{2j},&d_{K^\tau}(x)\ge\tau+\rho-\rho\left(\frac1{2j}\right)^{1/\be},
    \end{cases}
\end{align*}
for which  we have $F(D^2w_\sharp)\ge h^j(w_\sharp)$ in $\{\tau<d_{K^\tau}<\tau+\rho\}$, $F\left(D^2\left(\frac1{2j}\right)\right)=0=h^j\left(\frac1{2j}\right)$, and $\frac1{2j}=\max\left(w_\sharp,\frac1{2j}\right)$ in $\left\{\tau+\rho-\rho\left(\frac1{2j}\right)^{1/\be}<d_{K^\tau}<\tau+\rho\right\}$. This implies $F(D^2w_{\sharp,j})\ge h^j(w_{\sharp,j})$ in $\R^n\setminus K$. As we clearly have $w_{\sharp,j}=w_\sharp=1$ on $\p K=\{d_{K^\tau}=\tau\}$ and $w_{\sharp,j}=\frac1{2j}$ in $\left\{d_{K^\tau}\ge\tau+\rho-\rho\left(\frac1{2j}\right)^{1/\be}\right\}$, $w_{\sharp,j}$ is a subsolution of \eqref{eq:sol-const-reg} with $w_{\sharp,j}=\frac1{2j}$ in $B_{2R_0}\setminus\overline{B_{R_0}}$ for any $R_0>1$.\\

Next, we construct a supersolution of \eqref{eq:sol-const-reg}. To this aim, we take an open ball $B_{r_0}$ such that $r_0>1$ and $B_{r_0/2}\supset K$. For a large constant $R_0\gg r_0$ to be chosen later, we consider a continuous function $w^\sharp_j:\R^n\setminus K\to\R$ defined by 
\begin{align}\label{eq:supersol}
    w^\sharp_j(x):=\begin{cases}
    1,&x\in B_{r_0}\setminus K,\\
    \left(1-\frac1{2j}\right)\left(\frac{R_0-|x|}{R_0-r_0}\right)^4+\frac1{2j},&x\in B_{R_0}\setminus \overline{B_{r_0}},\\
    \frac1{2j},&x\in\R^n\setminus \overline{B_{R_0}}.
    \end{cases}
\end{align}
Note that $w^\sharp_j$ is $C^2$ in $\R^n\setminus\overline{B_{r_0}}$, in particular, across $\p B_{R_0}$. A direct computation yields
 \begin{align*}
    D^2w^\sharp_j(x)=4\left(1-\frac1{2j}\right)\frac{(R_0-|x|)^2}{(R_0-r_0)^4}\left[\left(\frac{R_0+2|x|}{|x|^3}\right)x\otimes x+\left(\frac{|x|-R_0}{|x|}\right)I\right],\quad x\in B_{R_0}\setminus \overline{B_{r_0}}.
\end{align*}
Thus, $$
F(D^2w^\sharp_j)=F(0)=0\quad\text{on }\p B_{R_0},
$$
and for $R_0\gg r_0$ $$
F(D^2w^\sharp_j)\le \cM^+_{\la,\Lambda}(D^2w^\sharp_j)\le\frac{C(n,\la,\Lambda)}{R_0}\quad\text{in }B_{R_0}\setminus\overline{B_{r_0}}.
$$
Note that $\frac1{2j}\le w^\sharp_j\le 1$ in $B_{R_0}\setminus\overline{B_{r_0}}$. Since $1/2\le h^j(t)\le 1$ for $\frac1{2j}\le t\le \e_1$ and $h^j(t)=h(t)$ for $t>\e_1$, we have that $\inf_{\left[\frac1{2j},1\right]}h^j$ is bounded below by a positive constant, independent of $j$. Thus we can find a large constant $R_0>0$, independent of $j$, such that $F(D^2w^\sharp_j)\le \inf_{\left[\frac1{2j},1\right]}h^j\le h^j(w^\sharp_j)$ in $B_{R_0}\setminus \overline{B_{r_0}}$. Since $h^j\ge0$, we further have $F(D^2w_j^\sharp)\le h^j(w^\sharp_j)$ in $\R^n\setminus\overline{B_{r_0}}$. Moreover, we clearly have $F(D^2w^\sharp_j)=0\le h^j(w^\sharp_j)$ in $B_{r_0}\setminus K$. Thus, to prove that $w^\sharp_j$ is a supersolution of \eqref{eq:sol-const-reg}, it is enough to show $F(D^2w^\sharp_j)\le h^j(w^\sharp_j)$ in $B_{R_0}\setminus\overline{B_{r_0/2}}$.

To this aim, we let $w^\sharp_{j,1}(x):=1$ and $w^\sharp_{j,2}(x):=\left(1-\frac1{2j}\right)\left(\frac{R_0-|x|}{R_0-r_0}\right)^4+\frac1{2j}$ for $x\in B_{R_0}\setminus\overline{B_{r_0/2}}$. 
By taking $R_0$ larger if necessary, but still independent of $j$, we can show that $F(D^2w^\sharp_{j,2})\le \inf_{\left[\frac1{2j},1\right]}h^j$ in $B_{R_0}\setminus\overline{B_{r_0/2}}$ by arguing as we did for $w^\sharp_j$ above. Since we also have $F(D^2w^\sharp_{j,1})=0\le \inf_{\left[\frac1{2j},1\right]}h^j$ and $w^\sharp_j=\min\{w^\sharp_{j,1}, w^\sharp_{j,2}\}$ in $B_{r_0}\setminus\overline{B_{r_0/2}}$, we infer that $F(D^2w^\sharp_j)\le\inf_{\left[\frac1{2j},1\right]}h^j\le h^j(w^\sharp_j)$ in $B_{R_0}\setminus\overline{B_{r_0/2}}$.

Now, we will use $w_{\sharp,j}$ and $w^\sharp_j$ to find a solution to \eqref{eq:sol-const-reg}.
 We let $\mu^j>0$  be such that $2\|\D h^j\|_{L^\infty((0,1))}<\mu^j$, and  construct 
  a sequence of functions $\{w^j_k\}^\infty_{k=0}$ defined in $B_{2R_0}\setminus K$ as follows:  set $w^j_0:=w_{\sharp,j}$, and for $k\ge0$ let $w^j_{k+1}$ be the unique solution of \begin{align*}
    \begin{cases}
    F(D^2w^j_{k+1})-\mu w^j_{k+1}+\mu w^j_k-h^j(w^j_k)=0&\text{in }B_{2R_0}\setminus K,\\
    w^j_{k+1}=1&\text{on }\p K,\\
    w^j_{k+1}=\frac1{2j}&\text{on }\p B_{2R_0}.
    \end{cases}
\end{align*}
Following the argument in the proof of Theorem~2.1 in \cite{RicTei11}, we can show that $$
w_{\sharp,j}=w^j_0\le w^j_1\le\cdots\le w^j_k\le w^j_{k+1}\le \cdots \le w^\sharp_j,
$$
and that the limit $v^j(x):=\lim_{k\to\infty}w^j_k(x)$ exists and is a solution of \eqref{eq:sol-const-reg}.

Since $h^j$ (the r.h.s. of \eqref{eq:sol-const-reg})
has uniform bound, independent of $j$, the limit $v:=\lim_{j\to\infty}v^j$ exists over a subsequence in $B_{2R_0}\setminus K$ and it satisfies $F(D^2v)=h(v)$ in $B_{2R_0}\setminus K$. Clearly, $v$ is nonnegative and $v=0$ in $B_{2R_0}\setminus\overline{B_{R_0}}$. Moreover, since $F(D^2v^j)=h^j(v^j)\ge0$ in $B_{2R_0}\setminus K$ and $v^j\le 1$ on $\p(B_{2R_0}\setminus K)$, we have $v^j\le 1$ in $B_{2R_0}\setminus K$, by the maximum principle. Taking $i\to\infty$ and extending $v=0$ in $\R^n\setminus \overline{B_{2R_0}}$, we get $v\le 1$ in $\R^n\setminus K$. On the other hand, from $v^j\ge w_{\sharp,j}\ge w_\sharp$ in $B_{2R_0}\setminus K$ and $w_\sharp=0$ on $\R^n\setminus \overline{2R_0}$, we have $v^j\ge w_\sharp$ in $\R^n\setminus K$, and thus $v\ge w_\sharp$ in $\R^n\setminus K$. In sum, we have $w_\sharp\le v\le 1$ in $\R^n\setminus K$, and since $w_\sharp=1$ on $\p K$, we get $v=1$ on $\p K$. Therefore, $v$ is a compactly-supported nonnegative solution to \eqref{eq:sol-const}.

\medskip\noindent \emph{Step 2.} 
To prove the quasi-concavity of $v$, let $v^*$ be the quasi-concave envelope of $v$. 
Then, by Proposition~\ref{prop:v*-subsol} in Appendix~\ref{appen:quasi-concave-subsol},
 $v^*$ is a subsolution to  \eqref{eq:sol-const}.
Since $v^*$ is a quasi-concave function, it suffices  to show  $v^*\equiv v$. For this purpose, we use Lavrentiev Principle. Without loss of generality, we assume $0\in K^{\mathrm{o}}$ and define for $t>0$ \begin{align*}
&D:=\{x\in\R^n:v(x)>0\},\quad D^*:=\{v^*>0\}\\ &v_t(x):=v(tx),\quad  D_t:=\{x:tx\in D\}=\{x:v_t(x)>0\}.
\end{align*}
Let
$$
E:=\{0<t<1\,:\,v_t(x)\ge v^*(x)\text{ for all } x\in D^*\}.
$$
The fact that $v^*\le\|v\|_{L^\infty(\R^n)}=1$ in the bounded set $D^*$ and $v_t=1$ in $t^{-1}K:=\{x:tx\in K\}$, together with $0\in K^{\mathrm{o}}$, implies that $v_t\ge v^*$ in $D^*$ for $t>0$ small so that $D^*\subset t^{-1}K$. This gives $E\neq\emptyset$. Now, towards a contradiction, we assume $t_0:=\sup E<1$.

We claim that there is a point $x^0\in\overline{D^*}\setminus K$ such that $v_{t_0}(x^0)=v^*(x^0)$. Assume, towards a contradiction, that  $v_{t_0}>v^*$ in $\overline{D^*}\setminus K$. Then $v_{t_0}>v^*$ in a smaller compact set $\overline{D^*}\setminus (t_0^{-1}K)^{\mathrm{o}}$. Since $v^*=0$ on $\p D^*$ and $v_{t_0}=1$ in $t_0^{-1}K$, by continuity we can find $t_1\in(t_0,1)$ slightly larger than $t_0$ such that $v_{t_1}\ge v^*$ in $\overline{D^*}\setminus (t_1^{-1}K)^{\mathrm{o}}$. As $v^*\le 1$ in $\R^n$ and $v_{t_1}=1$ in $t_1^{-1}K$, we further have $v_{t_1}\ge v^*$ in $\overline{D^*}$, which contradicts the definition of $t_0$.

Now we divide the  proof into two cases \begin{align*}
    &\mathbf{A}: \quad  x^0\in(D^*\cap D_{t_0})\setminus K\quad(\text{or }v^*(x^0)=v_{t_0}(x^0)>0),\\
    &\mathbf{B}:\quad  x^0\in\p D^*\cap\p D_{t_0}\quad(\text{or }v^*(x^0)=v_{t_0}(x^0)=0).
\end{align*}
In Case $\mathbf{A}$, we have $h(v^*(x^0))=h(v_{t_0}(x^0))>t_0^2h(v_{t_0}(x^0))>0$, and by continuity, $h(v^*)>t_0^2h(v_{t_0})$ in a neighborhood of $x^0$, say in $B_\rho(x^0)$. This gives $F(D^2v^*)\ge h(v^*)>t_0^2h(v_{t_0})=F(D^2v_{t_0})$ there. Therefore, $\cM^-_{\la,\Lambda}(D^2(v_{t_0}-v^*))\le F(D^2v_{t_0})-F(D^2v^*)<0$ in $B_\rho(x^0)$, and  (by the definition of $t_0$)  $v_{t_0}-v^*\ge 0$ in $B_\rho(x^0)$. Since $v_{t_0}-v^*$ attains a local minimum at $x^0$, the strong minimium principle implies $v_{t_0}-v^*\equiv0$ in $B_\rho(x^0)$. This contradicts $\cM^-_{\la,\Lambda}(D^2(v_{t_0}-v^*))<0$.

Next, we consider the case $\mathbf{B}$, i.e.,  $x^0\in \p D^*\cap \p D_{t_0}$. Note that $D^*\subset D_{t_0}$ and $v_{t_0}\ge v^*$ in $D^*$. By continuity, $v^*\le v_{t_0}\le \e_1$ in $B_\rho(x^0)\cap D^*$ for some small $\rho>0$, where $\e_1$ is as in \eqref{eq:assump-h}. Since $h(v^*)=h(v_{t_0})=1$ in $B_\rho(x^0)\cap D^*$, we can proceed as in case $\mathbf A$ to get $\cM^-_{\la,\Lambda}(D^2(v_{t_0}-v^*))<0$ in $B_\rho(x^0)\cap D^*$. This implies $v_{t_0}>v^*$ in $B_\rho(x^0)\cap D^*$, otherwise $v_{t_0}-v^*$ has a local minimum and we can argue as in case $\mathbf A$ to reach a contradiction.

We claim that for any small $\eta>0$, we can construct a cone ${\mathcal C}_{x^0}\subset \R^n$ with vertex $x^0$ and its Lipschitz norm less than $\eta$ such that ${\mathcal C}_{x^0}\cap B_\rho(x^0)\subset D^*$ for some $\rho>0$. Indeed, the point $x^0\in \p D^*$ can be written as a convex combination of $x^1,\cdots, x^k$, with $k\le n$, such that $x^i\in \p D$, $1\le i\le k$, and there exists a hyperplane $\Pi$ supporting $\p D^*$ at $x^0$ and $\p D$ at $x^i$, $1\le i\le k$, see e.g., \cite{BiaLonSal09}. Without loss of generality, we may assume $e_n=(0,\cdots,0,1)$ is normal to $\Pi$ and points towards $\p D^*$ at $x^0$ and $\p D$ at $x^i$, $1\le i\le k$. Note that from \eqref{eq:assump-h}, near each point $x^i$, the function $v$ can be seen as a solution of \begin{align*}
    F(D^2v)=\chi_{\{v>0\}}\,\,\text{ and }\,\, v\ge 0\quad \text{in }D\cap B_r(x^i).
\end{align*}
By the result of \cite{Lee98}, the free boundary $\p D$ is $C^1$ near the point $x^i$, thus we can construct a cone ${\mathcal C}_{x^i}$, with direction $e_n$ and its Lipschitz norm less than $\eta$, such that ${\mathcal C}_{x^i}\cap B_\rho(x^i)\subset D\subset D^*$ for small $\rho>0$. Since $D^*$ is convex and $x^0$ is a convex combination of $x^1,\cdots,x^k$, we can find the desired cone ${\mathcal C}_{x^0}$ satisfying ${\mathcal C}_{x^0}\cap B_\rho(x^0)\subset D^*$.

Now we let $w^\star$ be a solution of
\begin{align*}
    \begin{cases}
    \cM^-_{\la,\Lambda}(D^2w^\star)=0&\text{in }{\mathcal C}_{x^0}\cap B_{\rho}(x^0),\\
    w^\star=0&\text{on }\p {\mathcal C}_{x^0}\cap B_{\rho}(x^0),\\
    w^\star=v_{t_0}-v^*&\text{on }{\mathcal C}_{x^0}\cap \p B_{\rho}(x^0).
    \end{cases}
\end{align*}
Then the comparison principle yields $w^\star\le v_{t_0}-v^*$ in ${\mathcal C}_{x^0}\cap B_\rho(x^0)$. Thus, for any fixed $0<\al<1$, $\sup_{B_r(x^0)\cap \mathcal C_{x^0}}(v_{t_0}-v^*)\ge \sup_{B_r(x^0)\cap \mathcal C_{x^0}}w^\star\ge cr^{1+\al}$ for small $r>0$ by Lemma~3.5 in \cite{AllKriSha20}.

On the other hand, since $F(D^2v_{t_0}(x))=t_0^2F(D^2v(t_0x))=t_0^2h(v(t_0x))$ in $B_\rho(x^0)\cap D_{t_0}$ and $h$ is bounded, $v_{t_0}$ is $C^{1,\gamma}$ for all $\al<\gamma<1$, thus $\sup_{B_r(x^0)\cap \mathcal C_{x^0}}(v_{t_0}-v^*)\le \sup_{B_r(x^0)\cap \mathcal C_{x^0}}v_{t_0}\le Cr^{1+\gamma}$ for small $r>0$. This is a contradiction.
\end{proof}

Using the preceding lemma and the regularity result in Appendix~\ref{appen:grad-holder}, we prove Theorem~\ref{thm:quasi-concave} for the case the operator $L$ is the fully nonlinear operator $F$.

\begin{proof}[Proof of Theorem~\ref{thm:quasi-concave} for the fully nonlinear case]

\emph{Step 1.} For each large $i\in\N$ (such that $(1/2)^i<\e_1$), we consider a function $g^i:\R\to\R$ such that 
\begin{align}
    \label{eq:assump-f-approx}
    \begin{cases}
    g^i=0\text{ on }(-\infty,0],\\
    g^i\text{ is strictly positive, bounded and Lipschitz continuous on }(0,\infty),\\
    g^i(t)=1,\quad 0<t\le (1/2)^i,\\
    g^i\le g+2\text{ on $(0,\infty)$},\\
    g^i\to g \text{ uniformly on every compact subset of $(0,\infty)$}.
    \end{cases}
\end{align}
For each regularized problem of \eqref{eq:sol} \begin{align*}
    \begin{cases}
    F(D^2u^i)=g^i(u^i)&\text{in }\R^n\setminus K,\\
    u^i=1&\text{on }\p K,
    \end{cases}
\end{align*}
let $u^i$ be the solution constructed in Lemma~\ref{lem:quasi-concave-const}, which is compactly-supported, nonnegative and quasi-concave.

By Theorem~\ref{thm:grad-holder-F} in Appendix~\ref{appen:grad-holder}, for every bounded open set $D\Subset \R^n\setminus K$, $u^i$ is uniformly $C^{1,\be-1}_{\loc}(\overline{D})$ for large $i$ with
\begin{align}
    \label{eq:approx-unif-opt-reg}
    \sup_{B_r(x)\cap D}u^i\le C\left(r^\be+u^i(x)\right)\quad\text{for any }x\in D,
\end{align}
for some constant $C>0$, independent of $i$. This, combined with $0\le u^i\le 1$ in $\R^n\setminus K$, implies that over a subsequence $u:=\lim_{i\to\infty} u^i$
 exists in $\R^n\setminus K$, and $u$ is quasi-concave and satisfies $F(D^2u)=g(u)$ in $\{u>0\}$ and $|\D u|=0$ on $\p\{u>0\}$.

\medskip\noindent\emph{Step 2.} To prove $u=1$ on $\p K$, we let $w_\sharp$ be as in \eqref{eq:subsol}. In view of the end of \emph{Step 1} in the proof of Lemma~\ref{lem:quasi-concave-const}, we have $w_\sharp\le u^i$ in $\R^n\setminus K$. This holds for every $i$ since in the construction of $w_\sharp$ in \eqref{eq:subsol} two constants $\tau$ and $\rho$ depend only on $a$, $\la$, $\Lambda$, $K$. Thus we get $w_\sharp\le u$ in $\R^n\setminus K$. This, along with $u\le 1$ in $\R^n\setminus K$ and $w_\sharp=1$ on $\p K$, implies $u=1$ on $\p K$.

\medskip\noindent\emph{Step 3.} It remains to show $\Omega=\{u>0\} \subset\R^n\setminus K$ is bounded. Towards a contradiction we assume $\Omega$ is unbounded. If the component of $\Omega$ enclosing $K$ is bounded, then there is nothing to prove. Thus we may assume it is unbounded, and for simplicity further assume $\Omega$ is an unbounded connected set. We observe that the solution $u$ constructed above satisfies $u \leq 1$, and thus $F(D^2u)=g(u)\ge c_1u^b$ in $\Omega$. Now we use the idea in Lemma~4.1 in \cite{FotSha17} again to prove the non-degeneracy of $u$. Let $w:=u^{1-b}$ in $\Omega$, and compute $$
D^2w=(1-b)u^{-b}D^2u-\frac{b}{1-b}\cdot\frac{\D w\otimes\D w}{w}.
$$
Then it follows that $$
F(D^2w)\ge F((1-b)u^{-b}D^2u)+\cM_{\la,\Lambda}^-\left(\frac{b}{1-b}\cdot\frac{\D w\otimes\D w}{w}\right)\ge c_1(1-b)-\frac{b\Lambda}{1-b}\frac{|\D w|^2}w.
$$
For a point $x^0\in \Omega$ we set $h_{x^0}(x):=w(x)-\tilde c|x-x^0|^2$ for a small constant $\tilde c>0$, independent of $x^0$, to be determined later. Then
\begin{align*}
    \cM^+_{\la,\Lambda}(D^2h_{x^0})&\ge F(D^2w)-F(D^2(\tilde c|x-x^0|^2))\ge F(D^2w)-cM^{+}_{\la,\Lambda}(D^2(\tilde c|x-x^0|^2))\\
    &\ge c_1(1-b)-\left(\frac{b\Lambda}{1-b}\right)\frac{|\D w|^2}{w}-2\tilde cn\Lambda.
\end{align*}
This yields that in $\Omega=\{u>0\}=\{w>0\}$
\begin{align*}
    \cM^+_{\la,\Lambda}(D^2h_{x^0})+\frac{b\Lambda}{1-b}\left[\frac{\D(w+\tilde c|x-x^0|^2)}{w}\D h_{x^0}-\frac{4c}{w}h_{x^0}\right]&\ge c_1(1-b)-2\tilde c\Lambda\left(n+\frac{2b}{1-b}\right) \ge0,
\end{align*}
provided $\tilde c$ is small enough. Now, for every $R>0$, we choose a point $x^0\in\Omega$, so that $B_R(x^0)\subset\R^n\setminus K$. Note that $\frac{\D(w+\tilde c|x-x^0|^2)}w$ and $\frac{4c}w$ may not be bounded in $\Omega\cap B_R(x^0)$, which does not allow us to apply the maximum principle. To rectify this, we consider a subset $\{x\in\Omega:w(x)>\de\}$ of $\Omega$, with a constant $\delta\in(0,1)$, where the above two quotients are bounded. 
If $\delta>0$ is small enough, then the connected component of $\{w>\delta\}$ containing $x^0$, say $A$, satisfies $A\cap \p B_R(x^0)\neq\emptyset$. For the sake of simplicity we assume $A=\{w>\delta\}$, i.e., $\{w>\delta\}$ is connected. Now, if $\delta<\frac12w(x^0)$, then we have $h_{x^0}(x^0)=w(x^0)$ and $h_{x^0}=\de-\tilde c|x-x^0|^2<\frac12w(x^0)$ on $\p\{w>\de\}\cap B_R(x^0)$, thus $\sup_{\{w>\de\}\cap \p B_R(x^0)}h_{x^0}\ge w(x^0)>0$ by the maximum principle applied in $\{w>\de\}\cap B_R(x^0)$. 
This implies $\sup_{\Omega\cap B_R(x^0)}w\ge \sup_{\{w>\de\}\cap \p B_R(x^0)}w\ge \tilde cR^2$ for any $R>0$, which contradicts $w=u^{1-b}\le 1$.
\end{proof}


\subsection{$p$-Laplacian case}\label{subsec:p-lap}
In this subsection we prove Theorem~\ref{thm:quasi-concave}, for the case  when $L$ is the $p$-Laplacian $\Delta_pu$. Similar to the proof of  fully nonlinear version, we first consider (in the following lemma)
 the problem with $g$ replaced by $h$ (see \eqref{eq:assump-h}), and use it to prove Theorem~\ref{thm:quasi-concave}.

\begin{lemma}
\label{lem:quasi-concave-const-p}
Let $K\in\cA$, and let $H$ be as in \eqref{eq:H}, i.e. $H' = h$.
 Then there exists a minimizer of $$
J_{H}(w,\R^n\setminus K):=\int_{\R^n\setminus K}\left(|\D w|^p+pH(w)\right)
$$
over $\cK_K:=\{w\in W^{1,p}(\R^n\setminus K):w=1\text{ on }\p K\}$, that is compactly-supported, nonnegative and quasi-concave.
\end{lemma}

\begin{remark}
To prove Lemma~\ref{lem:quasi-concave-const-p}, we first prove the existence of a minimizer of the functional $J^\e_H(\cdot,\R^n\setminus K)$, $0<\e<1$, as in \eqref{eq:energy-ep} with desired properties. In \eqref{eq:energy-ep}, the extra term $-\e^{p/2}$ is needed, as $\left(\e+|\D w|^2\right)^{p/2}+pH(w)\ge \e^{p/2}$ and the integrand is taken in $\R^n\setminus K$ with infinite measure. An alternative way of removing $-\e^{p/2}$ is taking the integral in $B_{R_0}$ with large $R_0>0$, as we will prove the non-degeneracy and the boundedness of the support in Step 3 in the lemma below.
\end{remark}

\begin{proof}
\emph{Step 1.} 
For small $\e\in(0,1)$, we minimize  the  energy functional
\begin{align}
\label{eq:energy-ep}
J^\e_{H}(w,\R^n\setminus K):=\int_{\R^n\setminus K}\left(\left(\e+|\D w|^2\right)^{p/2}-\e^{p/2}+pH(w)\right),
\end{align}
over   $\cK_K$. Note that $J^\e_H(w,\R^n\setminus K)<\infty$ whenever $w\in\cK_K$ has a compact support. A minimizer $v^\e$ exists due to the semi-continuity of $J^\e$, and it is a solution of 
\begin{align}
    \label{eq:sol-const-p}
    \begin{cases}
    \ddiv\left[\left(\e+|\D v^\e|^2\right)^{\frac p2-1}\D v^\e\right]=h(v^\e)&\text{in }\R^n\setminus K,\\
    v^\e=1&\text{on }\p K.
    \end{cases}
\end{align}
The first equation in \eqref{eq:sol-const-p} is equivalent to $$
a^{ij}(\D v^\e)v^\e_{x_ix_j}=h(v^\e)\text{ in }\R^n\setminus K,
$$
where $a^{ij}(z)=\left(\e+|z|^2\right)^{p/2-2}\left[(p-2)z_iz_j+\delta_{ij}\left(\e+|z|^2\right)\right]$. Note that for every $\xi\in\R^n$ and $z\in \R^n$, 
\begin{align}
    \label{eq:ellip}
    \mu_1\left(\e+|z|^2\right)^{p/2-1}|\xi|^2\le a^{ij}(z)\xi_i\xi_j\le \mu_2\left(\e+|z|^2\right)^{p/2-1}|\xi|^2,
\end{align}
where $\mu_1=\min(p-1,1)$, $\mu_2=\max(p-1,1)$ (see e.g. \cite{Lew83}).

Since $H(t)=0$ for $t\le0$, we have $J^\e_H(\max\{v^\e,0\}, \R^n\setminus K)\le J^\e_H(v^\e,\R^n\setminus K)$. Moreover, $\max\{v^\e,0\}\in \mathcal{K}_K$ and the strict inequality $J^\e_H(\max\{v^\e,0\}, \R^n\setminus K)< J^\e_H(v^\e,\R^n\setminus K)$ holds when the set $\{v^\e<0\}$ is nonempty. This implies that $v^\e\ge0$ in $\R^n\setminus K$. Similarly, we can show that $v^\e\le 1$ in $\R^n\setminus K$ as well.


Next, we claim that $v^\e$ has  compact support. Suppose this fails, and that  $\{v^\e>0\}$ is  unbounded, and hence   for each $R> 0$ we have  a point $x^0\in\{v^\e>0\}$ such that $B_R(x^0)\subset \R^n\setminus \tilde K$. 
Notice that $\{v^\e>0\}$ can have exactly one component as $v^\e$ is an energy minimizer of $J_H^\e(\cdot,\R^n\setminus K)$. That is, $\{v^\e>0\}$ is a connected unbounded set.
We then fix a compact set $\tilde K\supset K$ strictly larger than $K$. Since $0\le v^\e\le 1$ in $\R^n\setminus K$, Lemma~1 in \cite{ChoLew91} gives $\sup_{\R^n\setminus \tilde K}|\D v^\e|<\infty$. Thus, in view of \eqref{eq:ellip}, the linear operator $
I_{v^\e}$ defined by $
I_{v^\e}((b_{ij})_{n\times n})=a^{ij}(\D v^\e)b_{ij}$ is uniformly elliptic in $\R^n\setminus\tilde K$. 
$$
\tilde v^\e(x):=v^\e(x)-v^\e(x^0)-\delta|x-x^0|^2,\quad x\in B_R(x^0),
$$
for a constant $\delta>0$, independent of $R$, to be determined later. From the assumption \eqref{eq:assump-h} on $h$, we can infer that $h(v^\e)\ge c_2\chi_{\{v^\e>0\}}$ for some $c_2>0$. Then we have in $\{v^\e>0\}\cap B_R(x^0)$ \begin{align*}
    I_{v^\e}(D^2\tilde v^\e)=I_{v^\e}(D^2v^\e)-I_{v^\e}(D^2(\delta|x-x^0|^2))\ge c_2-\delta \sum_{i=1}^na^{ii}(\D v^\e)>0,
\end{align*}
where the last inequality follows provided $\delta>0$ is small enough, independent of $R$. This is possible since $|\D v^\e|$ is bounded in $\R^n\setminus\tilde K$. The maximum principle tells us $\sup_{\p(\{v^\e>0\}\cap B_R(x^0))}\tilde v^\e\ge \tilde v^\e(x^0)=0$. On $B_R(x^0)\cap\p\{v^\e>0\}$, we have $\tilde v^\e=-v^\e(x^0)-\delta|x-x^0|^2<0$, thus $\sup_{\{v^\e>0\}\cap \p B_R(x^0)}\tilde v^\e\ge0$, which in turn gives $\sup_{\{v^\e>0\}\cap \p B_R(x^0)}v^\e\ge \delta R^2$. This implies that  $R$, chosen as above, is universally bounded, and hence so is the set  $\{v^\e>0\}$.


Next, we prove that $v^\e$ is quasi-concave. Note that $I_{v^\e}(0)=0$ and $I_{v^\e}$ is concave, homogeneous of degree $1$, and uniformly elliptic in $\R^n$ away from $K$. Moreover,  the quasi-concave envelope of $v^\e$, denoted by  $(v^\e)^*$,
satisfies  (see Proposition~\ref{prop:v*-subsol-p} in Appendix~\ref{appen:quasi-concave-subsol}) \begin{align}
    \label{eq:const-p-2nd-subsol}
    \begin{cases}
    I_{v^\e}(D^2(v^\e)^*) \geq h((v^\e)^*)&\text{in }\R^n\setminus K,\\
    (v^\e)^*=1&\text{on }\p K.
    \end{cases}
\end{align}
With these properties of $v^\e$ and $(v^\e)^*$ at hand, the quasi-concavity of $v^\e$ can be obtained by repeating the argument in \emph{Step 2} in the proof of Lemma~\ref{lem:quasi-concave-const}.

\medskip\noindent \emph{Step 2.} 
We claim that a limit function of $v^\e$ (over a subsequence) is the  desired function in Lemma~\ref{lem:quasi-concave-const-p}. Indeed, by the result of \cite{ChoLew91} and  for some $0<\al<1$, 
$\{v^\e\}_{0<\e<1}$ are uniformly $C^{1,\al}$ in every compact subset of $\R^n\setminus K$. Thus, there exists a function $v$ such that over a subsequence $v^\e$, $\D v^\e$ converge uniformly to $v$, $\D v$, respectively, on compact subsets of $\R^n\setminus K$. Clearly, $v$ is nonnegative and quasi-concave in $\R^n\setminus K$.

To show that $v$ is a minimizer of $J_{H}(w,\R^n\setminus K)=\int_{\R^n\setminus K}\left(|\D w|^p+pH(w)\right)$ in $\cK_K$, let $w\in \cK_K$ be given. If $w$ has a compact support, then we have for every $R>0$ large 
\begin{align}
    \label{eq:energy-min}
    \begin{split}
    J_H(v,B_R\setminus K)&\le\liminf_{\e\to0}J^\e_H(v^\e,B_R\setminus K)\le \liminf_{\e\to0}J^\e_H(v^\e,\R^n\setminus K)\\
    &\le\liminf_{\e\to0}J^\e_H(w,\R^n\setminus K)=J_H(w,\R^n\setminus K).
\end{split}\end{align}
Here, in the first inequality we applied Fatou's lemma and in the last step we used $\left(\e+|\D w|^2\right)^{p/2}-\e^{p/2}+pH(w)\le \left[\left(1+|\D w|^2\right)^{p/2}+pH(w)\right]\chi_{\text{\{supp} (w)\}}$ and applied the dominated convergence theorem. Taking $R\to\infty$ in \eqref{eq:energy-min} and applying the monotone convergence theorem give $J_H(v,\R^n\setminus K)\le J_H(w,\R^n\setminus K)$. In addition, for each $0<\e<1$, we have $0\le v^\e\le 1$ and 
\begin{align*}
    \int_{B_R\setminus K}|\D v^\e|^p&\le J^\e_H(v^\e,B_R\setminus K)+\int_{B_R\setminus K}\e^{p/2}\le J^\e_H(w,B_R\setminus K)+\int_{B_R\setminus K}\e^{p/2}\\
    &\le \int_{B_R\setminus K}\left(\left(1+|\D w|^2\right)^{p/2}+1\right),
\end{align*}
thus $v^\e$ is uniformly bounded in $W^{1,p}(B_R\setminus K)$. This implies up to a subsequence $v^\e\to v$ weakly in $W^{1,p}(B_R\setminus K)$, and thus $v=1$ on $\p K$. Therefore, $v\in \mathcal{K}_K$.

We now consider the general case $w\in \cK_K$ without the compact support assumption. For $m\in\N$, let $\phi_m:\R^n\to[0,1]$ be a cut-off function defined by \begin{align}\label{eq:cut-off-ftn}
    \phi_m(x):=\begin{cases}
    1,&|x|\le m,\\
    m+1-|x|,&m<|x|<m+1,\\
    0,&|x|\ge m+1.
    \end{cases}
\end{align}
Then $w_m:=w\phi_m\in\mathcal{K}_K$ has a compact support and satisfies
\begin{align*}
    \int_{\R^n\setminus K}|\D w_m|^p&=\int_{B_{m+1}\setminus K}|\phi_m\D w+w\D \phi_m|^p\le \int_{B_m\setminus K}|\D w|^p+C(p)\int_{B_{m+1}\setminus B_m}\left(|\D w|^p+|w|^p\right)\\
    &\le \int_{\R^n\setminus K}|\D w|^p+C(p)\int_{B_{m+1}\setminus B_m}\left(|\D w|^p+|w|^p\right).
\end{align*}
Moreover, as $w_m\le w$ and $H$ is nondecreasing, $H(w_m)\le H(w)$. Thus, $$
J_H(v,\R^n\setminus K)\le J_H(w_m,\R^n\setminus K)\le J_H(w,\R^n\setminus K)+C(p)\int_{B_{m+1}\setminus B_m}\left(|\D w|^p+|w|^p\right).
$$
Taking $m\to\infty$ and using $w\in W^{1,p}(\R^n\setminus K)$, we obtain $J_H(v,\R^n\setminus K)\le J_{H}(w,\R^n\setminus K)$, and conclude that $v$ is an energy minimizer of $J_H(\cdot,\R^n\setminus K)$.

\medskip\noindent \emph{Step 3.} We claim that $v$ is compactly supported.
As before, we assume to the contrary $\{v>0\}$ is an unbounded connected set.
 Note that as $0\le v\le1$, $\Delta_pv=h(v)\ge c_1v^b$ in $\{v>0\}$. Now we will use the idea in Lemma~4.1 in \cite{FotSha17} to prove the following version of non-degeneracy property of $v$: there exists a constant $c_0>0$, depending only on $n$, $p$, $b$, $c_1$, such that for any $R>0$ \begin{align}
    \label{eq:nondeg-p}
    \sup_{B_R(x^0)\cap \{v>0\}}v\ge c_0R^{\frac{p}{p-1-b}}\quad\text{for some }x^0\in\R^n\setminus K.
\end{align}
To prove $\eqref{eq:nondeg-p}$, for $\al:=1-\frac{b}{p-1}\in(0,1)$ let $\bar v:=v^\al$. For a point $x^0\in\{v>0\}$, set $q_{x^0}(x):=\bar c|x-x^0|^{\frac p{p-1}}$ with a small constant $\bar c>0$, depending only on $n$, $p$, $b$, $c_1$, to be determined later. Using $\Delta_pv\ge c_1v^b$ in $\{v>0\}$ and $|\D \bar v|=\al v^{\al-1}|\D v|$, we can compute \begin{align*}
    \Delta_p\bar v&=\al^{p-1}v^{(\al-1)(p-1)}\Delta_pv+\al^{p-1}(\al-1)(p-1)v^{(\al-1)(p-1)-1}|\D v|^p\\
    &\ge \al^{p-1}c_1+\frac{(\al-1)(p-1)}{\al\bar v}|\D\bar v|^p.
\end{align*}
From $\Delta_pq_{x^0}=\bar c n\left(\frac{p}{p-1}\right)^{p-1}$ and $|\D q_{x^0}|^p=\bar c^{p-1}\left(\frac{p}{p-1}\right)^pq_{x^0}$, we further have \begin{align*}
    \Delta_p\bar v-\Delta_pq_{x^0}&\ge\al^{p-1}c_1-\bar cn\left(\frac{p}{p-1}\right)^{p-1}+\frac{(\al-1)(p-1)}{\al\bar v}|\D\bar v|^p\\
    &=\al^{p-1}c_1-\bar cn\left(\frac{p}{p-1}\right)^{p-1}+\frac{(\al-1)(p-1)}{\al\bar v}\left(|\D\bar v|^p-|\D q_{x^0}|^p\right)\\
    &\qquad+\frac{(\al-1)(p-1)\bar c^{p-1}\left(\frac{p}{p-1}\right)^p}{\al\bar v}q_{x^0}\\
    &=\al^{p-1}c_1-\bar cn\left(\frac{p}{p-1}\right)^{p-1}+\frac{(\al-1)(p-1)}{\al\bar v}\left(|\D\bar v|^p-|\D q_{x^0}|^p\right)\\
    &\qquad+\frac{(\al-1)(p-1)\bar c^{p-1}\left(\frac{p}{p-1}\right)^p}{\al\bar v}(q_{x^0}-\bar v)+\frac{(\al-1)(p-1)\bar c^{p-1}\left(\frac{p}{p-1}\right)^p}{\al}.
\end{align*}
Since the function $t\mapsto t^p$ is convex in $[0,\infty)$, we have $|\D\bar v|^p-|\D q_{x^0}|^p\le p|\D\bar v|^{p-1}|\D(\bar v-q_{x^0})|$. Now, letting $h_{x^0}:=\bar v-q_{x^0}$, we see that
\begin{align}
\label{eq:p-Lapl-diff}
\begin{split}
    &\Delta_p\bar v-\Delta_pq_{x^0}+\frac{(1-\al)(p-1)p|\D\bar v|^{p-1}|^{p-1}}{\al\bar v}|\D h_{x^0}|-\frac{(1-\al)(p-1)\bar c^{p-1}\left(\frac{p}{p-1}\right)^p}{\al\bar v}h_{x^0}\\
    &\qquad \ge\al^{p-1}c_1-\bar cn\left(\frac P{p-1}\right)^{p-1}-\frac{(1-\al)(p-1)\bar c^{p-1}\left(\frac{p}{p-1}\right)^p}{\al}>0,
\end{split}\end{align}
where the last step follows if $\bar c$ is small enough. Note that $\Delta_p\bar v-\Delta_pq_{x^0}$ can be written as (see e.g. the proof of Lemma~4.7 in \cite{AllKriSha20}) $$
\Delta_p\bar v-\Delta_pq_{x^0}=\ddiv(A(x)\D h_{x^0}),
$$
where the matrix $A(x)=(a_{ij}(x))_{n\times n}$ is given by $$
a_{ij}(x)=\int_0^1|\D\bar v(x)t+\D q_{x^0}(x)(1-t)|^{p-2}m_{ij}^t\,dt,
$$
with $$
m_{ij}^t=\delta_{ij}+(p-2)\frac{(\p_i\bar vt+\p_iq_{x^0}(1-t))(\p_j\bar vt+\p_jq_{x^0}(1-t))}{|\D\bar v(x)t+\D q_{x^0}(x)(1-t)|^2}.
$$
It is also shown in \cite{AllKriSha20} that for some constant $\mu=\mu(p)>0$ $$
\mu^{-1}a(x)|\xi|^2\le a_{ij}(x)\xi_i\xi_j\le\mu a(x)|\xi|^2\quad\text{for any }\xi\in\R^n,
$$
where $$
a(x)=\int_0^1|\D\bar v(x)t+\D q_{x^0}(x)(1-t)|^{p-2}\,dt.
$$
To prove \eqref{eq:nondeg-p}, let $R>0$ be given, and as before let  $x^0\in\{v>0\}$ such that $B_R(x^0)\subset \R^n\setminus K$ and $\D v(x^0)\neq0$. 
This is possible since $\{v>0\}$ is unbounded and, for $\Delta_pv\ge c_1v^b>0$, $v$ cannot be a constant function in any open set in $\{v>0\}$. Moreover, as $\D q_{x^0}\neq 0$ in $\R^n\setminus \{x^0\}$, we infer $|\D\bar v|+|\D q_{x^0}|>0$ in $\{v>0\}$. 
Next, as we did in \emph{Step 3} in Theorem~\ref{thm:quasi-concave} for the fully nonlinear case, we can assume that there is a small constant $\delta>0$ with connected $\{\bar v>\delta\}$ such that $\delta<\frac12\bar v(x^0)$ and $\{\bar v>\de\}\cap\p B_R(x^0)\neq\emptyset$.
We then have by \eqref{eq:p-Lapl-diff} 
$$
div(A\D h_{x^0})+\frac{(1-\al)(p-1)p|\D\bar v|^{p-1}|^{p-1}}{\al\bar v}|\D h_{x^0}|-\frac{(1-\al)(p-1)\bar c^{p-1}\left(\frac{p}{p-1}\right)^p}{\al\bar v}h_{x^0}>0
$$
and the two quotients $\frac{(1-\al)(p-1)p|\D\bar v|^{p-1}|^{p-1}}{\al\bar v}$ and $\frac{(1-\al)(p-1)\bar c^{p-1}\left(\frac{p}{p-1}\right)^p}{\al\bar v}$ are uniformly bounded in $\{\bar v>\delta\}\cap B_R(x^0)$. Furthermore, as $|\D\bar v|+|\D q_{x^0}|$ is uniformly bounded below and above by positive constants in $\{\bar v>\delta\}\cap B_R(x^0)$, so is $a(x)$ there, and thus $A=(a_{ij})_{n\times n}$ is uniformly elliptic in $\{\bar v>\delta\}\cap B_R(x^0)$. Hence, we can apply the maximum principle to obtain $\sup_{\p(\{\bar v>\delta\}\cap B_R(x^0))}h_{x^0}\ge h_{x^0}(x^0)=\bar v(x^0)$. Since $h_{x^0}\le\bar v=\delta<\frac12\bar v(x^0)$ on $B_R(x^0)\cap\p\{\bar v>\delta\}$, we infer that $\sup_{\{\bar v>\delta\}\cap\p B_R(x^0)}h_{x^0}\ge \bar v(x^0)>0$. As $q_{x^0}=\bar c R^{\frac{p}{p-1}}$ on $\p B_R(x^0)$, we further have $\sup_{\{v>0\}\cap B_R(x^0)}\bar v\ge\sup_{\{\bar v>\delta\}\cap\p B_R(x^0)}\bar v\ge \bar cR^{\frac{p}{p-1}}$. This gives \eqref{eq:nondeg-p}, and contradicts $0\le v\le 1$ in $\R^n\setminus K$, unless $R$ can  not be chosen too  large, which is the desired conclusion.
\end{proof}

Now we prove the existence of a solution to \eqref{eq:sol} when the operator $L$ is $p$-Laplacian. In fact, for $G(t):=\int_{-\infty}^tg(s)\,ds$, we find an energy minimizer of $$
J_G(u,\R^n\setminus K):=\int_{\R^n}\left(|\D u|^p+PG(u)\right)
$$
over $\cK_K$, which is compactly-supported, nonnegative and quasi-concave. Clearly, this minimizer solves \eqref{eq:sol}.

\begin{proof}[Proof of Theorem~\ref{thm:quasi-concave} for $p$-Laplacian case]For functions $g^i$ satisfying \eqref{eq:assump-f-approx}, we define $G^i(t):=\int_{-\infty}^tg^i(s)\,ds$, $-\infty<t<\infty$. By Lemma~\ref{lem:quasi-concave-const-p}, there is a minimizer $u^i$ of $$
J_{G^i}(w,\R^n\setminus K)=\int_{\R^n\setminus K}\left(|\D w|^p+pG^i(w)\right)
$$
over $\cK_K=\{w\in W^{1,p}(\R^n\setminus K):w=1\text{ on }\p K\}$, which is compactly-supported, nonnegative and quasi-concave. We can use the regularity results in Appendix~\ref{appen:grad-holder} (Theorem~\ref{thm:grad-holder-p-geq2} for $2\le p<\infty$ and Theorem~\ref{thm:grad-holder-reg-p<2} for $1<p<2$) to see that $\{u^i\}$ are uniformly $C^{1,\al}$ in compact subsets of $\R^n\setminus K$. Thus, over a subsequence, $u^i\to u$ in $\R^n\setminus K$ for some function $u$. Clearly, $u$ is nonnegative and quasi-concave. 

To show that $u=1$ on $\p K$, we take a bounded and compactly-supported function $w$ in $\cK_K$. As $0\le g^i\le g+2$, we have
$$
\int_{\R^n\setminus K}|\D u^i|^p\le J_{G^i}(u^i,\R^n\setminus K)\le J_{G^i}(w,\R^n\setminus K)\le \int_{(\operatorname{supp}w)\setminus K}\left(|\D w|^p+p(G(w)+2w)\right)<\infty.
$$
Since $0\le u^i\le 1$ by the maximum principle, we have that for any $R>0$ $\{u^i\}$ is uniformly bounded in $W^{1,p}(B_R\setminus K)$, and infer $u=\lim_{i\to\infty}u^i=1$ on $\p K$.

By Fatou's lemma, it is easy to see that $u$ is an energy minimizer of $J_G(u,\R^n\setminus K)=\int_{\R^n\setminus K}\left(|\D u|^p+pG(u)\right)$ over $\mathcal K_K$. Thus, it satisfies $\Delta_pu=g(u)$ in $\{u>0\}\setminus K$. Moreover, the nonnegativity and the $C^{1,\al}$-regularity of $u$ in $\R^n\setminus K$ imply that $|\D u|=0$ on $\p\{u>0\}$.

Finally, the compact support of $\{u>0\}$ can be obtained in a similar way as we did in \emph{Step 3} in the proof of Lemma~\ref{lem:quasi-concave-const-p}, where the compact support of a solution $v$ of \eqref{eq:sol-const} is proved.
\end{proof}


\section{Proof of Theorem~\ref{thm:quasi-concave-thin}}

In this section, we prove Theorem~\ref{thm:quasi-concave-thin}. As in the previous sections, we first consider its regularized version, Lemma~\ref{lem:quasi-concave-const-thin}-\ref{lem:quasi-concave-const-p-thin}, and then use them to obtain Theorem~\ref{thm:quasi-concave-thin}.

\subsection{Fully nonlinear case}

\begin{lemma}\label{lem:quasi-concave-const-thin}
Let $K\in\tilde{\cA}$ and $h:\R\to\R$ be a function satisfying \eqref{eq:assump-h}.
Then, there exists a solution of the problem
\begin{align}
    \label{eq:sol-const-thin}
    \begin{cases}
    F(D^2v)=h(v)&\text{in }\R^n_+,\\
    v=1&\text{on }K^{\mathrm{o}},\\
    v=0&\text{on }\R^{n-1}\setminus K,
    \end{cases}
\end{align}
which is a compactly-supported, nonnegative and quasi-concave.
\end{lemma}

\begin{proof}
The proof follows the  argument used in the proof of Lemma~\ref{lem:quasi-concave-const}. 

\medskip\noindent\emph{Step 1.}
For a large constant $R_0>0$ and functions $h^j$ as in \eqref{eq:h^j}, we find solutions $v^j:B_{2R_0}^+\to[0,1]$ of
\begin{align}
    \label{eq:sol-const-thin-reg}
    \begin{cases}
    F(D^2v^j)=h^j(v^j)&\text{in }B_{2R_0}^+,\\
    v^j=1&\text{on }K^{\mathrm{o}},\\
    v^j=\frac1{2j}&\text{on }(B_{2R_0}'\setminus K)\cup (\p B_{2R_0})^+,
    \end{cases}
\end{align}
with $v^j=\frac1{2j}$ in $B^+_{2R_0}\setminus \overline{B^+_{R_0}}$.
 We let $\tau_m=1/m$, $m\in\N$, and define subsets of $K$ by
$$
K^{\tau_m}:=\{x\in K\,:\, d_{\p_{\R^{n-1}}K}(x)\ge\tau_m\},
$$
where $\p_{\R^{n-1}}K$ is the boundary of $K$ relative to $\R^{n-1}$. Note that $K^{\tau_m}\neq\emptyset$ when $m$ is large, say $m\ge M$ for some $M\in\N$. We then consider its translation $$\tilde{K}^{\tau_m}:=K^{\tau_m}-\tau_mx_n=\{x-\tau_mx_n\,:\, x\in K^{\tau_m}\}\subset \{x_n=-\tau_m\}.
$$
For small $\rho_m>0$, to be determined later, we define a function $w_\sharp^m:\R^n_+\to[0,\infty)$ by \begin{align}\label{eq:subsol-thin}
    w_\sharp^m(x):=\begin{cases}
    \frac1{\rho_m^\be}\left[\tau_m+\rho_m-d_{\tilde K^{\tau_m}}(x)\right]^\be,& \tau_m<d_{\tilde K^{\tau_m}}(x)<\tau_m+\rho_m,\\
    0,& d_{\tilde K^{\tau_m}}(x)\ge \tau_m+\rho_m.
    \end{cases}
\end{align}
Note that $w_\sharp^m$ is an analogue of \eqref{eq:subsol}, and $w^m_\sharp=1$ on $K^{\tau_m}$. Since $d_{\tilde{K}^{\tau_m}}(x)\ge\sqrt{2}\tau_m$ for every $x\in\R^{n-1}\setminus K$, we have $w_\sharp^m=0$ on $\R^{n-1}\setminus K$, provided $\rho_m<(\sqrt{2}-1)\tau_m$. In view of \emph{Step 1} in the proof of Lemma~\ref{lem:quasi-concave-const}, we can see that $F(D^2w_\sharp^m)\ge h^j(w_\sharp^m)$ in $\{\tau_m<d_{\tilde K^{\tau_m}}<\tau_m+\rho_m\}\cap\R^n_+$ if $\rho_m>0$ is small enough, depending only on $a$, $\la$, $\Lambda$, $\tau_m$, $K$. Defining a function $w_{\sharp,j}^m:\R^n_+\to[1/(2j),\infty)$ by
 \begin{align*}
    w_{\sharp,j}^m(x):=\begin{cases}
    w_\sharp^m(x),& \tau_m<d_{\tilde K^{\tau_m}}(x)<\tau_m+\rho_m-\rho_m\left(\frac1{2j}\right)^{1/\be},\\
    \frac1{2j},& d_{\tilde K^{\tau_m}}(x)\ge \tau_m+\rho_m-\rho_m\left(\frac1{2j}\right)^{1/\be}
    \end{cases}
\end{align*}
and arguing as in the Lemma~\ref{lem:quasi-concave-const}, we can show that $w^m_{\sharp,j}$ solves for any $R_0>1$
\begin{align*}
    \begin{cases}
    F(D^2w_{\sharp,j}^m)\ge h^j(w_{\sharp,j}^m)&\text{in }B_{2R_0}^+,\\
    w_{\sharp,j}^m=1&\text{on }K^{\tau_m},\\
    \frac1{2j}\le w_{\sharp,j}^m\le1&\text{on }K\setminus K^{\tau_m},\\
    w_{\sharp,j}^m=\frac1{2j}&\text{on }(B_{2R_0}'\setminus K)\cup(\p B_{2R_0})^+,
    \end{cases}
\end{align*}
with $w^m_{\sharp,j}=\frac1{2j}$ in $B^+_{2R_0}\setminus \overline{B^+_{2R_0}}$.
 As $K^{\tau_m}\nearrow K^{\mathrm{o}}$, the function $w_{\sharp,j}:\R^n_+\to[0,1]$ defined by 
$$
w_{\sharp,j}:=\sup_{m\ge M}w_{\sharp,j}^m
$$
becomes a subsolution of \eqref{eq:sol-const-thin-reg}, with $w_{\sharp,j}=\frac1{2j}$ in $B^+_{2R_0}\setminus \overline{B^+_{R_0}}$, for any $R_0>1$.

Next, we construct a supersolution of \eqref{eq:sol-const-thin-reg}. 
Let $w^\sharp_j$ be as in \eqref{eq:supersol} with $K\subset \R^{n-1}$. In view of the proof of Lemma~\ref{lem:quasi-concave-const}, for large $R_0$, $w^\sharp_j$ restricted to $B^+_{2R_0}$ satisfies
\begin{align*}
    \begin{cases}
    F(D^2w^\sharp_j)\le h^j(w^\sharp_j)&\text{in }B^+_{2R_0},\\
    w^\sharp_j=1&\text{on }K,\\
    \frac1{2j}\le w^\sharp_j\le 1&\text{on }B'_{2R_0}\setminus K,\\
    w^\sharp_j=\frac1{2j}&\text{on }(\p B_{2R_0})^+,
    \end{cases}
\end{align*}
with $w_j^\sharp=\frac1{2j}$ in $B_{2R_0}^+\setminus \overline{B_{R_0}^+}$.
 Notice that $w^\sharp_j$ may not be a supersolution of \eqref{eq:sol-const-thin-reg} as it may be strictly greater than $\frac1{2j}$ on $B'_{2R_0}\setminus K$. To rectify this, let $w_j^\star: B^+_{2R_0}\to\R$ be a solution of \begin{align*}
    \begin{cases}
    F(D^2w_j^\star)=0&\text{in }B_{2R_0}^+,\\
    w_j^\star=1&\text{on }K,\\
    w_j^\star=\frac1{2j}&\text{on }(B_{2R_0}'\setminus K)\cup (\p B_{2R_0})^+.
    \end{cases}
\end{align*}
Note that $0\le w_j^\star\le 1$ in $B^+_{2R_0}$ by maximum/minimum principle, and set $\tilde w^\sharp_j:=\min\{w^\sharp_j,w_j^\star\}$ in $B^+_{2R_0}$. Then it is easy to see that $\tilde w_j^\sharp$ is a supersolution of \eqref{eq:sol-const-thin-reg}.

Now, to find a solution of \eqref{eq:sol-const-thin-reg} by using $w_{\sharp,j}$ and $\tilde w^\sharp_j$ constructed above, we take a constant $\mu^j>0$ such that $2|\D h^j|\le\mu^j$. We then make a sequence of functions $\{w_k^j\}^\infty_{k=0}$ defined in $B_{2R_0}^+$as follows: let $w_0^j=w_{\sharp,j}$, and for $k\ge0$ let $w^j_{k+1}$ be the unique solution to \begin{align*}
    \begin{cases}
    F(D^2w^j_{k+1})-\mu w^j_{k+1}+\mu w^j_k-h^j(w^j_k)=0&\text{in }B^+_{2R_0},\\
    w^j_{k+1}=1&\text{on }K,\\
    w^j_{k+1}=\frac1{2j}&\text{on }(\p B_{2R_0}^+)\setminus K.
    \end{cases}
\end{align*}
By using the argument in \cite{RicTei11}, we can see that $$
w_{\sharp,j}=w_0^j\le w_1^j\le\cdots\le w^j_k\le w^j_{k+1}\le\cdots \le \tilde w^\sharp_j.
$$
Then, the limit function $v^j(x):=\lim_{k\to\infty}w^j_k(x)$ satisfies \eqref{eq:sol-const-thin-reg}. Moreover, since $h^j$ is uniformly bounded, up to a subsequence $v:=\lim_{j\to\infty}v^j$ exists in $B_{2R_0}^+$ and it clearly satisfies $F(D^2v)=h(v)$ in $B^+_{2R_0}$ and $v=0$ in $B_{2R_0}^+\setminus\overline{B_{R_0}^+}$. For the boundary value of $v$ on $B_{R_0}'$, we note that $w_\sharp\le v^j$ in $B_{2R_0}^+$. In addition, we let $v^\star:B_{2R_0}^+\to\R$ be a solution to \begin{align*}
    \begin{cases}
    F(D^2v^\star)=0&\text{in }B_{2R_0}^+,\\
    v^\star=1&\text{on }K\cup(\p B_{2R_0})^+.\\
    v^\star=0&\text{on }B_{2R_0}'\setminus K.
    \end{cases}
\end{align*}
Then, for each $j$, $F(D^2v^\star)\le F(D^2v^j)$ in $B_{2R_0}^+$ and $v^\star\ge v^j$ on $\p(B_{2R_0}^+)$, thus $v^\star\ge v^j$ in $B_{2R_0}^+$ by comparison principle. In sum we have $w_\sharp\le v^j\le v^\star$, $v_\sharp=v^\star=1$ on $K^{\mathrm{o}}$ and $v_\sharp=v^\star=0$ on $B'_{2R_0}\setminus K$. This yields $v=1$ on $K^{\mathrm{o}}$ and $v=0$ on $B'_{2R_0}\setminus K$. Extending $v=0$ in $\R^n_+\setminus\overline{B_{2R_0}^+}$, $v$ becomes a compactly-supported nonnegative solution of \eqref{eq:sol-const-thin}.

\medskip\noindent\emph{Step 2.} Our next objective is to show the quasi-concavity of $v$.
 Since the quasi-concave envelope $v^*$ (of $v$) is a subsolution of \eqref{eq:sol-const-thin} by Proposition~\ref{prop:v*-subsol} in Appendix~\ref{appen:quasi-concave-subsol}, it suffices to show that $v\equiv v^*$. As in Lemma~\ref{lem:quasi-concave-const}, we will use Lavrentiev Principle. We assume without loss of generality that $0\in K$ and denote for $t>0$
\begin{align}\label{eq:Lavrentiev-notation}
\begin{split}
    &D:=\{x\in\overline{\R^n_+}:v(x)>0\},\qquad D^*:=\{x\in\overline{\R^n_+}:v^*(x)>0\},\\
    &v_t(x):=v(tx),\qquad D_t:=\{x\in\overline{\R^n_+}:tx\in D\}=\{x\in\overline{\R^n_+}:v_t(x)>0\}.
\end{split}\end{align}
Let also $$
E:=\{0<t<1\,:\,v_t(x)\ge v^*(x)\text{ for all }x\in D^*\}.
$$ 
We first claim that $E$ is a nonempty set. Indeed, the assumption that K has a nonempty $(n-1)$-dimensional interior gives that it contains a thin ball $B_{\rho_0}'=B_{\rho_0}\cap\R^{n-1}$ for some $\rho_0>0$. We then clearly have $v_t\ge v^*$ on $D^*\cap\R^{n-1}$ for small $t>0$. To extend the inequality $v_t\ge v^*$ to $D^*$ for small $t>0$, we observe that $F(D^2v^*)\ge h(v^*)>0$ in the positivity set of $v^*$, in particular near $K$. Thus, by the strong maximum principle, $v^*<1$ in $\R^n_+$ near $K$. Moreover, by applying the maximum principle to $v$, we get $v\le 1$ in $\R^n_+$, thus $v^*\le 1$ in $\R^n_+$. This gives $v^*=v=1$ on $K^{\mathrm{o}}$. Now, as the flat boundary $\p(\R^n_+)=\R^{n-1}$ satisfies the interior sphere condition \emph{uniformly} on $K^{\mathrm{o}}$, Hopf's  Lemma tells us that there is a positive constant $c_1>0$ such that for every $x^0\in K^{\mathrm{o}}$ $$
\limsup_{\R^n_+\ni x\to x^0}\frac{v^*(x)-v^*(x^0)}{|x-x^0|}<-c_1.
$$
This implies that $v^*(x)\le 1-c_1x_n$ in $[0,s)\times K^{\mathrm{o}}$ for some $c_1>0$ and $s>0$.

On the other hand, for $\rho_0>0$ small so that $B'_{\rho_0}\Subset K^{\mathrm{o}}$, since $\p_{x_n}v$ is continuous in $B_{\rho_0}^+\cup B_{\rho_0}'$ (and hence bounded there), we obtain that $v(x)\ge 1-c_2x_n$ in $B_{\rho_0}^+\cup B_{\rho_0}'$ for some $c_2>0$. This, combined with the above estimate $v^*\le 1-c_1x_n$ near $K$, implies that $v_t\ge v^*$ in $D^*$ for small $t>0$. Therefore, $E\neq\emptyset$.

Notice that the quasi-concavity of $v$ follows once we show $\sup E=1$. Towards a contradiction suppose $t_0:=\sup E\in(0,1)$. Then we claim that there is a point $x^0\in \overline{D^*}\cap \R^{n}_+$ such that $v_{t_0}(x^0)=v^*(x^0)$. Indeed, if we assume to the contrary that the claim is not true, then $v_{t_0}>v^*$ in $\overline{D^*}\cap \R^n_+$. By the definition of $t_0$, for each $t\in(t_0,1)$ the set $\{x\in D^*:v_t(x)<v^*(x)\}$ is nonempty. For such $t$, we take a connected component $A_t$ of $\{x\in D^*\,:\, v_t(x)<v^*(x)\}$. Since $D^*$ is bounded, we can find a sequence $\{t_j\}\subset(t_0,1)$ such that $t_j\searrow t_0$ and $A_{t_j}$ converges to a nonempty set, say $A$, in $\overline{D^*}$. As $v_{t_0}=v^*$ in $A$, the assumption $v_{t_0}>v^*$ in $\overline{D^*}\cap\R^n_+$ yields $A\subset\R^{n-1}$. From the observation 
$$v^*=\begin{cases}
1&\text{on }K^{\mathrm{o}},\\
0&\text{on }\R^{n-1}\setminus K
\end{cases}\quad\text{ and }\quad v_{t_0}=\begin{cases}
1&\text{on }t_0^{-1}K^{\mathrm{o}},\\
0&\text{on }\R^{n-1}\setminus t_0^{-1}K,
\end{cases}
$$ and 
we infer that either $A\subset \R^{n-1}\setminus t_0^{-1}K$ or $A\subset K^{\mathrm{o}}$.

If $A\subset\R^{n-1}\setminus t_0^{-1}K$, then using that $h(s)=1$ for $0<s<\e_1$, we have for $t>t_0$ close to $t_0$ $$
F(D^2v_t)=t^2h(v_t)<h(v_t)=1=h(v^*)\le F(D^2v^*)\quad\text{in }A_t.
$$
Then \begin{align*}
    \begin{cases}
    \cM^+_{\la,\Lambda}(D^2(v^*-v_t))\ge F(D^2v^*)-F(D^2v_t)>0\quad\text{in }A_t,\\
    v^*-v_t>0\quad\text{in }A_t,\\
    v^*-v_t=0\quad\text{on }\p A_t.
    \end{cases}
\end{align*}
This is a contradiction by the maximum principle.

Next, we consider the case when $A\subset K^{\mathrm{o}}$. From the continuity and the positivity of $h$ on $(0,\infty)$, there is a constant $\delta\in(0,1)$ such that if $s_1,s_2\in(1-\delta,1)$, then $\frac{h(s_1)}{h(s_2)}\ge\left(\frac{t_0+1}2\right)^2$. Thus, for $t>t_0$ close to $t_0$ $$
F(D^2v_t)=t^2h(v_t)\le t^2\left(\frac2{t_0+1}\right)^2h(v^*)<h(v^*)\le F(D^2v^*)\quad\text{in }A_t.
$$
As in the previous case, this yields a contradiction, and the claim is proved.

Now, we have $x^0\in\overline{D^*}\cap\R^n_+$ with $v_{t_0}(x^0)=v^*(x^0)$, thus we can split our proof into the following two cases 
\begin{align*}
    &\mathbf{A}.\ x^0\in(D^*\cap D_{t_0})\setminus \R^{n-1},\\
    &\mathbf{B}.\ x^0\in(\p D^*\cap\p D_{t_0})\setminus\R^{n-1}.
\end{align*}
In each case, we can argue as we did in Lemma~\ref{lem:quasi-concave-const} to get a contradiction. This completes the proof.
\end{proof}

\begin{proof}[Proof of Theorem~\ref{thm:quasi-concave-thin} for the Fully nonlinear case] \emph{Step 1.} For the proof, we follow the argument in the proof of Theorem~\ref{thm:quasi-concave}, with the help of Lemma~\ref{lem:quasi-concave-const-thin}. Let $g^i$ be as in \eqref{eq:assump-f-approx}, and apply Lemma~\ref{lem:quasi-concave-const-thin} to obtain a compactly-supported, nonnegative quasi-concave solution $u^i$ to \begin{align*}
    \begin{cases}
    F(D^2u^i)=g^i(u^i)&\text{in }\R^n_+,\\
    u^i=1&\text{on }K^{\mathrm{o}},\\
    u^i=0&\text{on }\R^{n-1}\setminus K.
    \end{cases}
\end{align*}
Repeating the argument in the proof of Theorem~\ref{thm:quasi-concave}, we have over a subsequence $u:=\lim_{i\to\infty}u^i$ exists in $\R^n_+$.

It is easy to see that $u$ is quasi-concave, $F(D^2u)=g(u)$ in $\{u>0\}$ and $|\D u|=0$ on $\p\{u>0\}\cap\R^n_+$. Moreover, each $u^i$ constructed in Lemma~\ref{lem:quasi-concave-const-thin} satisfies $0\le u^i\le1$ in $\R^n_+$, which gives $0\le u\le1$ in $\R^n_+$. As in \emph{Step 3} in the proof of Theorem~\ref{thm:quasi-concave} for the case $L=F$, we can prove the non-degeneracy property of $u$, which combined with $0\le u\le1$ implies the boundedness of $\{u>0\}$. Thus, it remains to show $u=1$ on $K^{\mathrm{o}}$ and $u=0$ on $\R^{n-1}\setminus K$.

\medskip\noindent\emph{Step 2.} To show that $u=1$ on $K^{\mathrm{o}}$, for each $m\in\N$, let $w^m_\sharp$ be as in \eqref{eq:subsol-thin}. In view of \emph{Step 1} in the proof of Lemma~\ref{lem:quasi-concave-const-thin}, $w_\sharp^m\le u^i$ in $\R^n_+$ for $m\ge M$. This holds for every $i$ since $\tau_m=1/m$ and $\rho_m$ in \eqref{eq:subsol-thin} as well as $M$ depend only on $a$, $\la$, $\Lambda$, $K$. We then have $w^m_\sharp\le u$ in $\R^n_+$ for $m\ge M$, which along with $u\le1$ in $\R^n_+$ and $w^m_\sharp=1$ on $K^{\tau_m}$ gives $u=1$ on $K^{\tau_m}$. Taking $m\to\infty$, we get $u=1$ on $K^{\mathrm{o}}$.

Finally, to prove $u=0$ on $\R^{n-1}\setminus K$, we fix $R>0$ large so that $B'_R\supset K$, and let $u_R^\star:B_R^+\to\R$ be a solution of \begin{align*}
    \begin{cases}
    F(D^2u_R^\star)=0&\text{in }B_R^+,\\
    u_R^\star=1&\text{on } K\cup(\p (B_R^+))^+,\\
    u_R^\star=0&\text{on }B_R'\setminus K.
    \end{cases}
\end{align*}
Then, applying the comparison principle yields $u^i\le u_R^\star$ in $B_R^+$. Taking $i\to\infty$, we get $u\le u_R^\star$ in $B_R^+$, and thus $u=0$ on $B_R'\setminus K$. Letting $R\to\infty$, we obtain $u=0$ on $\R^{n-1}\setminus K$. This completes the proof.
\end{proof}


\subsection{$p$-Laplacian case}

\begin{lemma}
\label{lem:quasi-concave-const-p-thin}
Suppose $K\in\tilde{\cA}$ and let $H$ be as in \eqref{eq:H}. Then there is an energy minimizer of 
$$
J_H(w,\R^n_+):=\int_{\R^n_+}\left(|\D w|^p+pH(w)\right)
$$
over $\tilde{\mathcal{K}}_K:=\{w\in W^{1,p}(\R^n_+)\,:\,w=1\text{ on }K,\,\,\, w=0\text{ on }\R^{n-1}\setminus K\}$,
which is compactly-supported, nonnegative and quasi-concave.
\end{lemma}

\begin{proof}
The lemma  can be proven  by following the lines in the proof of Lemma~\ref{lem:quasi-concave-const-p} and Lemma~\ref{lem:quasi-concave-const-thin}. We  thus only give a sketch of the proof.

For small $\e>0$, let $v^\e$ be a minimizer of $$
J^\e_H(w,\R^n_+)=\int_{\R^n_+}\left((\e+|\D w|^2)^{p/2}-\e^{p/2}+pH(w)\right)
$$
over $\tilde{\mathcal{K}}_K$. Then $v^\e$ solves \begin{align}
    \label{eq:sol-const-p-thin}
    \begin{cases}
        \ddiv\left[(\e+|\D v^\e|^2)^{\frac p2-1}\D v^\e\right]=h(v^\e)&\text{in }\R^n_+,\\
        v^\e=1&\text{on }K,\\
        v^\e=0&\text{on }\R^{n-1}\setminus K.
    \end{cases}
\end{align}
The first equation in \eqref{eq:sol-const-p-thin} is equivalent to $$
a^{ij}(\D v^\e)v^\e_{x_ix_j}=h(v^\e)\quad\text{in }\R^n_+,
$$
where $a^{ij}(z)=\left(\e+|z|^2\right)^{p/2-2}\left[(p-2)z_iz_j+\delta_{ij}\left(\e+|z|^2\right)\right]$. Here, $a^{ij}$ satisfies the ellipticity \eqref{eq:ellip}. Moreover, the fact that $\max\{v^\e,0\}$ and $\min\{v^\e,1\}$ are contained in $\tilde{\mathcal{K}}_K$ implies $0\le v^\e\le1$.

To show that $v^\e$ has a compact support, we suppose towards a contradiction $\{v^\e>0\}$ is an unbounded connected set. By Lemma~1 in \cite{ChoLew91}, we have $\sup_{\{x_1>1\}}|\D v^\e|<\infty$, and thus the operator $I_{v^\e}$, defined by $I_{v^\e}((b_{ij})_{n\times n})=a^{ij}(\D v^\e)b_{ij}$, is uniformly elliptic in $\{x_1>1\}$. Then, we can find a small constant $\de>0$ such that for any $R>1$ and $x^0\in\{v^\e>0\}$ with $B_R(x^0)\Subset\{x_n>1\}$,
$$
\sup_{\{v^\e>0\}\cap\p B_R(x^0)}v^\e\ge\de R^2.
$$
This non-degeneracy property of $v^\e$ contradicts its bound $0\le v^\e\le1$, and proves that $v^\e$ has a compact support.

Next, we show that $v^\e$ is quasi-concave by following the argument in \emph{Step 2} in the proof of Lemma~\ref{lem:quasi-concave-const-thin}. Indeed, Proposition~\ref{prop:v*-subsol-p} in Appendix~\ref{appen:quasi-concave-subsol} tells us its quasi-concave envelope $(v^\e)^*$ is a subsolution of \begin{align}
    \label{eq:const-p-2nd-subsol-thin}
    \begin{cases}
    I_{v^\e}(D^2(v^\e)^*)=h((v^\e)^*)&\text{in }\R^n_+,\\
    (v^\e)^*=1&\text{on }K,\\
    (v^\e)^*=0&\text{on }\R^{n-1}\setminus K.
\end{cases}
\end{align}
Assume $0\in K$ and for $t>0$ define (in analogy with \eqref{eq:Lavrentiev-notation}) \begin{align*}
    &D^\e:=\{x\in\overline{\R^n_+}:v^\e(x)>0\},\quad (D^\e)^*:=\{x\in\overline{\R^n_+}:(v^\e)^*>0\},\\
    &v^\e_t(x):=v^\e(tx),\quad D^\e_t:=\{x\in\overline{\R^n_+}:tx\in D^\e\}=\{x\in \overline{\R^n_+}:v^\e_t>0\},
\end{align*}
and $$
E^\e:=\{0<t<1:v_t^\e\ge(v^\e)^*\text{ in }(D^\e)^*\}.
$$
Note that the subsolution $(v^\e)^*$ of \eqref{eq:const-p-2nd-subsol-thin} satisfies Hopf's Lemma at every $x^0\in K$. Besides, for $\rho_0>0$ small so that $B_{2\rho_0}'\subset K$,
$\p_{x_n}v^\e$ is continuous in $\overline{B_{\rho_0}^+}$, hence bounded there.  As we have seen in Lemma~\ref{lem:quasi-concave-const-thin}, this implies $v_t^\e\ge (v^\e)^*$ in $(D^\e)^*$ for small $t\in(0,1)$, and thus $E^\e\neq\emptyset$. Now, as it is sufficient to prove $\sup E^\e=1$ for the quasi-concavity of $v^\e$, we assume to the contrary $\sup E^\e=t_0\in(0,1)$. Then, using that $I_{v^\e}$ is elliptic in $\R^n_+$ and uniformly elliptic in every compact subset in $\R^n_+$, we can argue as in Lemma~\ref{lem:quasi-concave-const-thin} to get a contradiction.

Due to the result of \cite{ChoLew91}, we have for some $0<\al<1$ that $\{v^\e\}_{0<\e<1}$ are uniformly $C^{1,\al}$ in compact subsets of $\R^n_+$, thus $v^\e\to v$ in $C^1_{\loc}(\R^n_+)$ for some function $v$. Clearly, $v$ is nonnegative and quasi-concave in $\R^n_+$. Moreover, we can proceed as in \emph{Step $2$-$3$} in the proof of Lemma~\ref{lem:quasi-concave-const-p} to prove that $v\in \tilde{\cK}_K$ is a minimizer of $J_H(\cdot,\R^n_+)$ over $\tilde{\cK}_K$ and $v$ has a compact support. This finishes the proof.
\end{proof}


\begin{proof}[Proof of Theorem~\ref{thm:quasi-concave-thin} for the $p$-Laplacian case]
The proof is similar to the proof of Theorem~\ref{thm:quasi-concave}, thus we only give an idea of the proof. For $g^i$ and $G^i$ as before, we apply Lemma~\ref{lem:quasi-concave-const-p-thin} to find a minimizer $u^i$ of 
$$
J_{G^i}(w,\R^n_+)=\int_{\R^n_+}\left(|\D w|^p+pG^i(w)\right)
$$
over $\tilde{\cK}_{K}$, which is compactly-supported, nonnegative and quasi-concave. Then, Theorem~\ref{thm:grad-holder-p-geq2}-\ref{thm:grad-holder-reg-p<2} imply that up to a subsequence $u^i\to u$ in $C^1_{\loc}(\R^n_+)$ for some function $u\in C^{1,\al}_{\loc}(\R^n_+)$. Clearly, $u$ is nonnegative and quasi-concave, and satisfies $|\D u|=0$ on $(\p\Omega\setminus K)\cap\R^n_+$, where $\Omega=\{u>0\}$. In addition, we can argue as in the proof of Theorem~\ref{thm:quasi-concave} for the $p$-Laplacian case to prove that $u=1$ on $K$, $u=0$ on $\R^{n-1}\setminus K$ and $u$ is an energy minimizer in $\tilde{\cK}_K$. Finally, we can show as before, using non-degeneracy of $u$,  that $u$ has a compact support.
 This completes the proof.
\end{proof}


\section{Proof of Theorem~\ref{thm:quasi-concave-hyb}}

The purpose of this section is to prove Theorem~\ref{thm:quasi-concave-hyb}. It can be regarded as a corollary to Theorem~\ref{thm:quasi-concave} and Theorem~\ref{thm:quasi-concave-thin}, as its proof simply follows from the argument and the results in those two preceding theorems.

\begin{lemma}
\label{lem:quasi-concave-const-hyb}
Let $K\in\tilde{\cA}$ and suppose $h:\R\to\R$ satisfies \eqref{eq:assump-h}. Then there is a solution of 
\begin{align}
    \label{eq:sol-const-hyb}
    \begin{cases}
    F(D^2v)=h(v)&\text{in }\R^n\setminus K,\\
    v=1&\text{on }K^{\mathrm{o}},
    \end{cases}
\end{align}
which is a compactly-supported, nonnegative and quasi-concave.
\end{lemma}

\begin{proof}
\emph{Step 1.} Let $h^j$ be as in \eqref{eq:h^j} and $R>0$ be a large constant. We first find solutions $v^j:B_{2R_0}\setminus K\to\R$ of 
\begin{align}
    \label{eq:sol-const-hyb-reg}
    \begin{cases}
    F(D^2v^j)=h^j(v^j)&\text{in }B_{2R_0}\setminus K,\\
    h^j=1&\text{on }K^{\mathrm{o}},\\
    v^j=\frac1{2j}&\text{in }B_{2R_0}\setminus\overline{B_{R_0}}.
    \end{cases}
\end{align}
Let $w$ be a solution of \eqref{eq:sol-const-thin} obtained in Lemma~\ref{lem:quasi-concave-const-thin}, and consider its even reflection $w_\sharp:\R^n\to\R$ defined by \begin{align*}
    w_\sharp(x):=\begin{cases}
    w_\sharp(x),&x_n\ge0,\\
    w_\sharp(x',-x_n),&x_n<0.
    \end{cases}
\end{align*}
If $R_0>0$ is large so that $B_{R_0}\supset \text{supp} w_\sharp$, then it is easy to see that $w_{\sharp,j}:=\max\left\{w_\sharp,\frac1{2j}\right\}$ is a subsolution of \eqref{eq:sol-const-hyb-reg}. Besides, the function $w_j^\sharp$ as in \eqref{eq:supersol} (with $K\subset\R^{n-1}$) is a supersolution of \eqref{eq:sol-const-hyb-reg}. With those $w_{\sharp,j}$ and $w_j^\sharp$ at hand, we can repeat the argument in Lemma~\ref{lem:quasi-concave-const} to get a solution $v^j$ of \eqref{eq:sol-const-hyb-reg} and prove that $v:=\lim_{j\to\infty}v^j$ exists over a subsequence and, after extending $v=0$ in $\R^n\setminus \overline{B_{2R_0}}$, $v$ solves \eqref{eq:sol-const-hyb}.

\medskip\noindent\emph{Step 2.} We proceed as in \emph{Step 2} in the proof of Lemma~\ref{lem:quasi-concave-const} up to the definition $$
E:=\sup\{0<t<1\,:\, v_t>v^*\text{ in }D^*\}.
$$
Since $K$ is contained in the thin space $\R^{n-1}$, as we did in \emph{Step 2} in Lemma~\ref{lem:quasi-concave-const-thin} we can apply Hopf's Lemma to $v^*$ and use $\p_{x_n}v$ is bounded in a small ball $B_{\rho_0}$ to get $E\neq\emptyset$. Now we assume towards a contradiction $t_0:=\sup E<1$. We then claim that $v_{t_0}(x^0)=v^*(x^0)$ for some point $x^0\in\overline{D^*}\setminus K$. Indeed, if the claim is not true, then $v_{t_0}>v^*$ in $\overline{D^*}\setminus K$. Since $\{x\in D^*:v_t(x)<v^*(x)\}$ is nonempty for every $t_0<t<1$, if we take a connected component $A_t$ of $\{x\in D^*:v_t<v^*\}$, then $A_t$ converges to a nonempty set $A\subset\overline{D^*}$ as $t\to t_0$. Since $v_{t_0}\le v^*$ in $A$ and $v_{t_0}>v^*$ in $\overline{D^*}\setminus K$, we have $A\subset K$. Again, we argue as in \emph{Step 2} in Lemma~\ref{lem:quasi-concave-const-thin} to get a contradiction, and prove the claim.

Now we can divide the proof into two cases 
\begin{align*}
    &\mathbf{A}.\ x^0\in(D^*\cap D_{t_0})\setminus K,\\
    &\mathbf{B}.\ x^0\in \p D^*\cap\p D_{t_0},
\end{align*}
and repeat the argument in \emph{Step 2} in Lemma~\ref{lem:quasi-concave-const} to arrive at a contradiction and get $E=1$. This completes the proof.
\end{proof}

\begin{lemma}
\label{lem:quasi-concave-const-hyb-p}
Let $K$, $H$ be as in Lemma~\ref{lem:quasi-concave-const-p-thin}. Then there exists a minimizer of 
$$
J_H(w,\R^n):=\int_{\R^n}\left(|\D w|^p+pH(w)\right)
$$
over $\hat{\cK}_K:=\{w\in W^{1,p}(\R^n):w=1\text{ on }K\}$, which is compactly-supported, nonnegative and quasi-concave.
\end{lemma}

\begin{proof}
The proof can be obtained by repeating the argument in Lemma~\ref{lem:quasi-concave-const-p}, thus we briefly sketch the proof.

For $0<\e<1$, let $v^\e$ be an energy minimizer of $$
J_H^\e(w,\R^n):=\int_{\R^n}\left(\left(\e+|\D w|^2\right)^{p/2}-\e^{p/2}+pH(w)\right)
$$
over $\hat{\cK}_K$. Clearly, $v^\e$ solves
\begin{align}
    \label{eq:sol-const-hyb-p}
    \begin{cases}
    \ddiv\left[\left(\e+|\D v^\e|^2\right)^{p/2-1}\D v^\e\right]=h(v^\e)&\text{in }\R^n\setminus K,\\
    v^\e=1&\text{on }K.
    \end{cases}
\end{align}
Fix a compact set $\tilde K\subset\R^n$ strictly larger than $K$. Then $\sup_{\R^n\setminus\tilde K}|\D v^\e|<\infty$, thus the first equation in \eqref{eq:sol-const-hyb-p} is an uniformly elliptic equation in $\R^n\setminus\tilde K$. With this at hand, we can argue as in Lemma~\ref{lem:quasi-concave-const-p} to prove that $v^\e$ has a compact support. Besides, the quasi-concavity of $v^\e$ can be derived by following the line of the proof for the corresponding result in Lemma~\ref{lem:quasi-concave-const-p-thin}. Finally, we can proceed as in \emph{Step 2-3} in Lemma~\ref{lem:quasi-concave-const-p} to complete the proof.
\end{proof}

We finish this section with a formal proof of our last main result.

\begin{proof}[Proof of Theorem~\ref{thm:quasi-concave-hyb}] 
With Lemma~\ref{lem:quasi-concave-const-hyb} and Lemma~\ref{lem:quasi-concave-const-hyb-p} at hand, the proof follows from repeating the argument in Theorem~\ref{thm:quasi-concave-thin}.
\end{proof}


\appendix

\section{Property of quasi-concave envelopes}\label{appen:quasi-concave-subsol}

In this appendix, we show that quasi-concave envelopes of regularized problems are subsolutions. They follow from applying the results and proofs in \cite{BiaLonSal09} to our settings. The fully nonlinear case is treated in Proposition~\ref{prop:v*-subsol}, and the $p$-Laplacian one can be found in Proposition~\ref{prop:v*-subsol-p}.

\begin{proposition}\label{prop:v*-subsol}
If $v$ is a solution of \eqref{eq:sol-const}, then its quasi-concave envelope $v^*$ is a subsolution of \eqref{eq:sol-const}. Similarly, if $v$ is a solution to \eqref{eq:sol-const-thin}, then $v^*$ is a subsolution for \eqref{eq:sol-const-thin}.
\end{proposition}

\begin{proof}
Proposition ~\ref{prop:v*-subsol} follows by applying Theorem~3.1 in \cite{BiaLonSal09}. Indeed, when $v$ is a solution to \eqref{eq:sol-const}, the theorem can be rephrased in our favor as the following: if $v$ is a solution of $\tilde F(v,D^2v)=0$ in $D\setminus K$, where $\tilde F(v,A):=F(A)-h(v)$ and $D=\{v>0\}$, then the quasi-concave envelope $v^*$ is a subsolution to $\tilde F(v,D^2v)=0$, provided the conditions $(C1)-(C5)$ below is satisfied
\begin{align*}
    &(C1)\quad \tilde{F}(v,A)\text{ is proper (i.e., }\tilde{F}(s,A)\le\tilde{F}(t,A)\text{ whenever } s\ge t),\\
    &(C2)\quad D\setminus K \text{ is a convex ring (i.e., }D\text{ and }K\text{ are convex)},\\
    &(C3)\quad \tilde F\text{ is continuous and degenerate elliptic in }(0,1)\times S(n)\\
    &\quad \qquad\text{(i.e., $\tilde F(v,A)\ge\tilde F(v,B)$ when $A\ge B$)},\\
    &(C4)\quad \text{For some $\al\in \R$ and for any fixed $t\in(t_0,t_1)$, the function $\Phi_{t,\al}(q,A):=q^\al\tilde F(t,\frac A{q^3})$}\\ 
    &\quad \qquad \text{is concave in $(0,\infty)\times S(n)$, }\\
    &(C5)\quad |\D v|>0\text{ in }D\setminus K.
\end{align*}
In fact, $(C1)-(C5)$ are assumed throughout theorems in \cite{BiaLonSal09}, and one can see from the proof of Thoerem~3.1 in \cite{BiaLonSal09} that it still holds without $(C1)-(C2)$. Moreover, the condition $(C3)$ simply follows from the properties of $F$ and $h$. For $(C4)$, we take $\al=3$ and see $\Phi_{t,3}(q,A)=q^3\tilde F(t,\frac{A}{q^3})=F(A)-h(t)q^3$ is concave as $F$ is concave and $h(t)\ge0$.

Now we will show that the theorem is still true without the condition $(C5)$ with small modification in its proof, since our function $v$ is a solution and hence admits the maximum principle. To see where $(C5)$ is used in \cite{BiaLonSal09}, let $\bar x\in D\setminus K$ such that $v^*(\bar x)>v(\bar x)$ and let $t=v^*(\bar x)$. Then for $D(t):=\{x:v(x)\ge t\}$ there exist $\la_i\in(0,1)$ and $x^i \in\p D(t)$, $1\le i\le n$, with $\Sigma_{i=1}^n\la_i=1$ such that $$
\bar x=\Sigma_{i=1}^n\la_ix^i \quad\text{and}\quad v^*(\bar x)=v(x^i)=t.
$$
Here, the condition $|\D v|>0$ is used in Theorem~3.1 in \cite{BiaLonSal09} to have the following:
    \begin{align*}
        & (P1)\quad t\in(0,1),\\
        &(P2)\quad \bar x\in\p D^*(t), \text{ where } D^*(t)=\{x:v^*(x)\ge t\},\\
        &(P3)\quad |\D v(x^i )|>0,\quad 1\le i\le n.
    \end{align*}
We claim that all of these three properties hold for the solution $v$ of \eqref{eq:sol-const} without the assumption $|\D v|>0$.

Indeed,  $(P1)$  follows from the definition $D=\{v>0\}$ and the fact that $v$ cannot
have a local maximum in $D\setminus K$, for $F(D^2v)=f(v)>0$ in $D\setminus K$.

$(P2)$ and $(P3)$ can be obtained by applying Proposition~1 and Proposition~2 in \cite{GreKaw09}, respectively. Indeed, these propositions require the following nondegeneracy property: \emph{If the level set $D(t)$, for $\inf_{D\setminus K}v<t<\sup_{D\setminus K}v$, has a supporting hyperplane passing through some point $x\in D(t)$, then $\D v(x)\neq0$.} This property holds for our solution $v$, since the supporting hyperplane guarantees that $\{v<t\}$ satisfies interior sphere condition at $x$ and thus Hopf's Lemma can be applied.

Next, we consider the second case when $v$ is a solution of \eqref{eq:sol-const-thin}. Recall that conditions $(C1)-(C2)$ are not essential. Moreover, $(C3)-(C4)$ and $(P1)-(P3)$ can be justified by the similar way as in the first case. This completes the proof.
\end{proof}

\begin{proposition}
\label{prop:v*-subsol-p}
If $v^\e$ is a solution of \eqref{eq:sol-const-p}, then its quasi-concave envelope $(v^\e)^*$ is a subsolution of \eqref{eq:const-p-2nd-subsol}. Similarly, if $v^\e$ solves \eqref{eq:sol-const-p-thin}, then $(v^\e)^*$ is a subsolution of \eqref{eq:const-p-2nd-subsol-thin}.
\end{proposition}

\begin{proof}
We only consider the case when $v^\e$ is a solution of \eqref{eq:sol-const-p}, as the other case can be treated in a similar way.

Since the operator $I_{v^\e}(B)=a^{ij}(\D v^\e)b_{ij}$, defined in the proof of Lemma~\ref{lem:quasi-concave-const-p},  satisfies  condition \eqref{eq:assump-fully-nonlinear} in every compact subset of $\R^n\setminus K$, we can prove that $(v^\e)^*$ is a subsolution by repeating the proof of Proposition~\ref{prop:v*-subsol}. We only need to check  condition $(C4)$ in Proposition~\ref{prop:v*-subsol} holds for $F_{v^\e}$. To prove it, let $\tilde I_{v^\e}(t,B):=a^{ij}(\D v^\e)b_{ij}-h(t)$, $B=(b_{ij})_{n\times n}$, $t\in(0,1)$, and define $\Phi_{t,\al}:=q^\al\tilde I_{v^\e}(t,\frac{B}{q^3})$ for some $\al\in \R$. If we take $\al=3$, then $\Phi_{t,3}(q,B)=q^3\left[a^{ij}(\D v^\e)\frac{b_{ij}}{q^3}-h(t)\right]=a^{ij}(\D v^\e)b_{ij}-h(t)q^3$ is a concave function, as desired.
\end{proof}


\section{Regularity of solutions and minimizers}\label{appen:grad-holder}

In this section we establish the local $C^{1,\al}$-regularity of solutions and energy minimizers in Lemma~\ref{lem:quasi-concave-const}-\ref{lem:quasi-concave-const-hyb-p}. We first obtain the regularity result when the operator $L$ is the fully nonlinear operator $F$.

\begin{theorem}
\label{thm:grad-holder-F}
For fixed $-1<a<0$ and $C_1>0$, let $\hat g:\R\to\R$ be a function satisfying $\hat g=0$ on $(-\infty,0]$ and $0\le \hat g(t)\le C_1t^a$ for $0<t<\infty$. Let $\Omega\subset\R^n$ be a bounded open set and $\be:=\frac2{1-a}\in(1,2)$. If $u$ is a nonnegative solution of $$
F(D^2u)=\hat g(u)\quad\text{in }\Omega,
$$
then $u\in C^{1,\be-1}_{\loc}(\Omega)$. Moreover, for every $\Omega'\Subset\Omega$ $$
\sup_{B_r(x)\cap\Omega'}u\le C(r^\be+u(x))\quad\text{for any }x\in A,\,r>0,
$$
where $C>0$ is a constant depending only on $n$, $a$, $\la$, $\Lambda$, $\Omega'$, but independent of $u$.
\end{theorem}

\begin{proof}
The proof follows the line of \cite{AraTei13}. Let $v(x):=u(x)^{1/\be}$, $x\in\{u>0\}$. Using $F(D^2u)=\hat g(u)$ and following the computation $(11)$-$(16)$ in \cite{AraTei13}, we can get in $\{u>0\}$
$$
F\left(D^2v+\frac{1+a}{1-a}v^{-1}\cdot\D v\otimes\D v\right)=fv^{-1},
$$
where $f(x)=\left(\frac{1-a}2\right)(u(x))^{-a}\hat g(u(x))$. Note that $f(x)\le\left(\frac{1-a}2\right)(u(x))^{-a}C_1u(x)^a=\frac{(1-a)C_1}2$, i.e., $f$ is bounded in $\{u>0\}$. With the above equation for $v$ and the boundedness of $f$ at hand, we can proceed as in \cite{AraTei13}, in particular Proposition~1 and Theorem~3, to obtain Theorem~\ref{thm:grad-holder-F}.
\end{proof}


Next, we derive the $C^{1,\al}$-regularity of energy minimizers concerning $p$-Laplacian operators treated in this paper. First, we make use of \cite{LeideQTei15} to get the result when $2\le p<\infty$.

\begin{theorem}\label{thm:grad-holder-p-geq2}
For fixed $-1<a<0$ and $C_1>0$, let $\hat g:\R\to\R$ be a function satisfying $\hg=0$ on $(-\infty,0]$ and $0\le \hg\le C_1t^a$ for $0<t<\infty$, and $\hG(t):=\int_{-\infty}^t\hg(s)\,ds$, $-\infty<t<\infty$. Let $\Omega\Subset\R^n$ be a bounded open set, $2\le p<\infty$ and $\varphi\in W^{1,p}(\Omega)\cap L^\infty(\Omega)$. If $u$ is a nonnegative energy minimizer of $$
J_{\hG}(u,\Omega)=\int_{\Omega}\left(|\D u|^p+\hG(u)\right)
$$
among all competitors $v\in W^{1,p}_0(\Omega)+\varphi$, then $u\in C^{1,\al}(\Omega)$ for some $0<\al<1$, depending only on $n$, $p$, $a$. Moreover, for every $\Omega'\Subset \Omega$, there exists a constant $C>0$ depending only on $\Omega'$, $n$, $p$, $a$, $\|\varphi\|_\infty$ such that $$
\|u\|_{C^{1,\al}(\Omega')}\le C.
$$
\end{theorem}

\begin{proof}
The proof of the theorem can be obtained by repeating the proof of the corresponding result in \cite{LeideQTei15}, whose consequence involves the $C^{1,\al}$-estimates of the nonnegative minimizer $v$ of $$
J_{\hG_a}(v,\Omega)=\int_{\Omega}\left(|\D v|^p+\hG_a(v)\right),
$$
where $\hG_a(v)=\int_{-\infty}^t\hg_a(s)\,ds$, $\hg_a(s)=C_1s^a\chi_{\{s>0\}}$. Indeed, in the proof for the regularity of $v$ in \cite{LeideQTei15}, the definition of $\hG_a$ is used only twice to have for any $j\in \N$ and $B_R\Subset\Omega$ \begin{align*}
    \int_{A_j}\left(\hG_a(v_j)-\hG_a(v)\right)\le0\quad\text{and}\quad \int_{B_R}\left(\hG_a(v_R^\star)-\hG_a(v)\right)\le C\int_{B_R}|v_R^\star-v|^\gamma,
\end{align*}
where $A_j:=\{v>j\}$, $v_j:=\min\{v,j\}$, $\gamma:=1+a\in(0,1)$, and $v_R^\star$ is the $p$-harmonic replacement of $u$ in $B_R$, i.e., $v_R^\star$ is the solution of the Dirichlet problem \begin{align*}
    \begin{cases}
    \Delta_pv_R^\star=0&\text{in }B_R,\\
    v_R^\star=v&\text{on }\p B_R.
    \end{cases}
\end{align*}

It is easy to see that the energy minimizer $u$ of $J_{\hG}(\cdot,\Omega)$ in our case also satisfies $$
\int_{A_j}\left(G(u_j)-G(u)\right)\le0\quad\text{and}\quad \int_{B_R}\left(G(u_R^\star)-G(u)\right)\le C\int_{B_R}|u_R^\star-u|^\gamma,
$$
where $u_R^\star$ is the $p$-harmonic replacement of $u$ in $B_R$.
In fact, the first inequality simply follows from the fact that $t\mapsto \hG(t)$ is nondecreasing and $u_j\le u$. For the second one, we assume $u_R^\star>u$ (otherwise, $\hG(u_R^\star)-\hG(u)\le 0$, and there is nothing to prove) and use Lemma~2.5 in \cite{LeideQTei15} to have $$
\hG(u_R^\star)-\hG(u)=\int_u^{u_R^\star}\hg(s)\,ds\le\int_u^hC_1s^a\,ds\le\frac{C_1}\gamma((u_R^\star)^\gamma-u^\gamma)\le\frac{C_1}\gamma(u_R^\star-u)^\gamma.
$$
This completes the proof.
\end{proof}

In the remaining of this section, we consider the case $1<p<2$ and prove an analogous regularity result with Theorem~\ref{thm:grad-holder-p-geq2}. We expect that similar results already exist in the literature, but it looks like that they are scattered and we were not able to find them. Thus, we give a complete proof of it. We start by proving a Hölder regularity with some degree $\delta=\delta(n,p)\in(0,1)$, following the line of \cite{DanPet05}.

\begin{lemma}
\label{lem:holder-reg-p}
For $1<p<2$, let $\hg$, $\hG$ be as in Theorem~\ref{thm:grad-holder-p-geq2}, $\Omega\Subset\R^n$ be a bounded open set, and $\varphi\in W^{1,p}(\Omega)\cap L^\infty(\Omega)$ be a nonnegative function. If $u$ is an energy minimizer of $$
J_{\hG}(u,\Omega)=\int_\Omega\left(|\D u|^p+\hG(u)\right)
$$
among all competitors $v\in W^{1,p}_0(\Omega)+\varphi$, then $u\in C^{0,\delta}(\Omega)$, for $\delta=\frac{p^2}{2n+p^2}\in(0,1)$. Moreover, for every $\Omega'\Subset\Omega$, there exists a constant $C>0$ depending only on $\Omega'$, $n$, $p$, $a$, $\|\varphi\|_\infty$ such that $$
\|u\|_{C^{0,\delta}(\Omega')}\le C.
$$
\end{lemma}

\begin{proof}
Note that $\max\{u,0\}\in W^{1,p}_0(\Omega)+\varphi$ is a valid competitor for $u$, and that $J_{\hG}(\max\{u,0\},\Omega)\le J_{\hG}(u,\Omega)$ as $\hg(t)=0$ for $(-\infty,0]$. Moreover, the strict inequality $J_{\hG}(\max\{u,0\},\Omega)< J_{\hG}(u,\Omega)$ holds when $\{u<0\}$ is nonempty. Thus, we should have $u\ge0$. Similarly, $\min\{u,\|\varphi\|_\infty\}\in W^{1,p}_0(\Omega)+\varphi$ , and $J_{\hG}(\min\{u,\|\varphi\|_\infty\},\Omega)<J_{\hG}(u,\Omega)$ if $\{u>\|\varphi\|_\infty\}\neq\emptyset$. Therefore, $0\le u\le \|\varphi\|_\infty$. In the proof, universal constant $C$ may vary, but will depend only on $n$, $p$, $a$, and $\|\varphi\|_\infty$.

We follow the argument in Section~3 in \cite{DanPet05} (in particular, Lemma~3.1). For any (small) ball $B_r(x^0)\Subset \Omega$, let $u_r^\star$ be the $p$-harmonic replacement of $u$ in $B_r(x^0)$.
Without loss of generality, we may assume $x^0=0$ and use $J_{\hG}(u, B_r)\le J_{\hG}(u_r^\star,B_r)$ and $\|u_r^\star\|_{L^\infty(B_r)}\le \|u\|_{L^\infty(B_r)}\le \|\varphi\|_\infty$ to have \begin{align}
    \label{eq:min-p-har-diff-est}
    \begin{split}
    \int_{B_r}\left(|\D u|^p-|\D u_r^\star|^p\right)&\le J_{\hG}(u,B_r)-J_{\hG}(u_r^\star,B_r)+\int_{B_r}(\hG(u_r^\star)-\hG(u))\\
    &\le \int_{B_r}\hG(u_r^\star)\le Cr^n,
\end{split}\end{align}
where in the last inequality we used $\hG(u_r^\star)=\int_{-\infty}^{u_r^\star}\hg(s)\,ds\le C(u_r^\star)^{1+a}\le C\|\varphi\|_\infty^{1+a}$. The following inequalities are obtained in the beginning of Section~3 in \cite{DanPet05}, only using that $u_r^\star$ is $p$-harmonic replacement of $u$ and $1<p<2$: \begin{align*}
    &\int_{B_r}\left(|\D u|^p-|\D u_r^\star|^p\right)\ge c\int_{B_r}|\D(u-u_r^\star)|^2\left(|\D u|+|\D u_r^\star|\right)^{p-2},\\
    &\int_{B_r}|\D(u-u_r^\star)|^p\le\left(\int_{B_r}|\D (u-u_r^\star)|^2(|\D u|+|\D u_r^\star|)^{p-2}\right)^{p/2}\left(\int_{B_r}\left(|\D u|+|\D u_r^\star|\right)^p\right)^{1-p/2}.
\end{align*}
Combining these two inequalities with $\int_{B_r}|\D u_r^\star|^p\le \int_{B_r}|\D u|^p$ and \eqref{eq:min-p-har-diff-est} gives \begin{align}
    \label{eq:min-p-har-grad-diff-est}
    \int_{B_r}|\D(u-u_r^\star)|^p&\le C\left(\int_{B_r}(|\D u|^p-|\D u_r^\star|^p)\right)^{p/2}\left(\int_{B_r}|\D u|^p\right)^{1-p/2}\\
    &\le Cr^{np/2}\left(\int_{B_r}|\D u|^p\right)^{1-p/2}.
\end{align}
Moreover, using again the fact that $h$ is the $p$-harmonic replacement of $u$ and applying Theorem~1 in \cite{Man86} (with $\e=0$), we get \begin{align}
    \label{eq:p-har-grad-sup-bound}
    \sup_{B_{r/2}}|\D u_r^\star|\le \left(\frac{C}{r^n}\int_{B_r}|\D u_r^\star|^p\right)^{1/p}\le \left(Cr^{-n}\int_{B_r}|\D u|^p\right)^{1/p}.
\end{align}
Now, for small $\e>0$ to be chosen below, we have for $0<r\le r_0(\e)$ with $r^\e\le 1/2$ 
\begin{align}\label{eq:min-Morrey-est}
    \int_{B_{r^{1+\e}}}|\D u|^p&\le C\left(\int_{B_{r^{1+\e}}}|\D(u-u_r^\star)|^p+\int_{B_{r^{1+\e}}}|\D u_r^\star|^p\right)\\
    &\le C\int_{B_r}|\D(u-u_r^\star)|^p+Cr^{(1+\e)n}\|\D u_r^\star\|_{L^\infty(B_{r/2})}^p\\
    &\le Cr^{np/2}\left(\int_{B_r}|\D u|^p\right)^{1-p/2}+Cr^{\e n}\int_{B_r}|\D u|^p,
\end{align}
where the last step follows from \eqref{eq:min-p-har-grad-diff-est} and \eqref{eq:p-har-grad-sup-bound}. Since $u$ is a $p$-subsolution ($\Delta_pu=\hg(u)\ge0$), we can apply Caccioppoli inequality (see e.g. Lemma~2.9 in \cite{Lin19}, whose proof works for $p$-subsolutions as well) to obtain \begin{align}
    \label{eq:caccio-ineq}
    \int_{B_r}|\D u|^p\le \frac{C}{r^p}\int_{B_{2r}}u^p\le Cr^{n-p}.
\end{align}
This, combined with \eqref{eq:min-Morrey-est}, yields (by letting $\rho=r^{1+\e}$) \begin{align*}
    \int_{B_\rho}|\D u|^p\le C\left(\rho^{\frac{n-p+p^2/2}{1+\e}}+\rho^{\frac{n-p+\e n}{1+\e}}\right).
\end{align*}
Taking $\e=\frac{p^2}{2n}$ so that $p^2/2=\e n$, we obtain that for $\delta=\frac{p^2}{2n+p^2}$ \begin{align}
    \label{eq:min-Morrey-est-improved}
    \int_{B_\rho}|\D u|^p\le C\rho^{n-p+p\delta}.
\end{align}
As this estimate holds for any center $x^0$ and any small radius $\rho$, by applying Morrey space embedding theorem, we can obtain $u\in C^{0,\delta}(\Omega)$ and $\|u\|_{C^{0,\delta}(\Omega')}\le C$ for every $\Omega'\Subset \Omega$.
\end{proof}

Now we use bootstrapping to improve the Hölder-continuity result in Lemma~\ref{lem:holder-reg-p} to almost Lipschitz regularity.

\begin{lemma} 
\label{lem:alm-Lip-reg-p}
For $1<p<2$, let $\hg$, $\hG$ $u$, $\varphi$ be as in Lemma~\ref{lem:holder-reg-p}. Then $u\in C^{0,\sigma}(\Omega)$ for every $0<\sigma<1$. Moreover, for any $\Omega'\Subset\Omega$, there exists a constant $C>0$ depending only on $\Omega'$, $n$, $p$, $a$, $\|\varphi\|_\infty$ such that $$
\|u\|_{C^{0,\sigma}(\Omega')}\le C.
$$
\end{lemma}

\begin{proof}
For $B_r\Subset\Omega$ and small $\e>0$ to be specified below, we have by \eqref{eq:min-Morrey-est} $$
\int_{B_{r^{1+\e}}}|\D u|^p\le Cr^{np/2}\left(\int_{B_r}|\D u|^p\right)^{1-p/2}+Cr^{\e n}\int_{B_r}|\D u|^p.
$$
Instead of \eqref{eq:caccio-ineq}, we use the improved estimate \eqref{eq:min-Morrey-est-improved}, $\int_{B_r}|\D u|^p\le Cr^{n-p+p\delta}$, to have for $\rho=r^{1+\e}$ $$
\int_{B_\rho}|\D u|^p\le C\rho^{\frac{n-p+p(\delta+p/2-p\delta/2)}{1+\e}}+C\rho^{\frac{n-p+\e n+p\delta}{1+\e}}.
$$
We take $\e=\frac{p^2}{2n}(1-\delta)$ so that the two terms on the right-hand side have the same power, and obtain $$
\int_{B_\rho}|\D u|^p\le Cr^{n-p+p\delta'},\quad \delta':=\delta+\frac{p/2(1-\delta)(1+p/n(1-\delta))}{1+\frac{p^2}{2n}(1-\delta)}.
$$
By iteration, we can find a sequence of positive numbers $\delta_1$, $\delta_2$, $\cdots$ such that $\delta_1=\delta$, $\delta_k<\delta_{k+1}\le \delta_k+\frac{p/2(1-\delta_k)(1+p/n(1-\delta_k))}{1+\frac{p^2}{2n}(1-\delta_k)}$ for $k\ge1$, and $\int_{B_\rho}|\D u|^2\le C_k\rho^{n-p+p\delta_k}$ for $0<\rho<\rho_k$ as long as $\delta_k<1$.
Note that if $0<\delta_k<\sigma$ for some $\sigma\in(0,1)$, then $$
\frac{p/2(1-\delta_k)(1+p/n(1-\delta_k))}{1+\frac{p^2}{2n}(1-\delta_k)}\ge \frac{p/2(1-\sigma)(1+p/n(1-\sigma))}{1+\frac{p^2}{2n}}.
$$
This implies that we can make $\delta_{k+1}-\delta_k$ greater than a universal positive constant, independent of $k$, as long as $\delta_k<1$. Thus, for any $\sigma\in(0,1)$ we can find $m\in\N$ such that $\sigma<\delta_m<1$ and $$
\int_{B_\rho}|\D u|^p\le C\rho^{n-p+p\delta_m}\le C\rho^{n-p+p\sigma}.
$$
As we have seen in Lemma~\ref{lem:holder-reg-p}, this Morrey-type estimate implies the $C^{0,\sigma}$-regularity.
\end{proof}

Finally, we prove the $C^{1,\al}$-regularity of energy minimizers with the help of Lemma~\ref{lem:alm-Lip-reg-p} and \cite{LeideQTei15}.

\begin{theorem}
\label{thm:grad-holder-reg-p<2}
Let $1<p<2$ and let $\hg$, $\hG$, $u$, $\varphi$ be as the above lemmas. Then there exists a constant $\al\in(0,1)$, depending only on $p$ and $a$ such that $u\in C^{1,\al}(\Omega)$. Furthermore, for every $\Omega'\Subset\Omega$, there exists a constant $C>0$, depending only on $\Omega'$, $n$, $p$, $a$, $\|\varphi\|_\infty$, such that $$
\|u\|_{C^{1,\al}(\Omega')}\le C.
$$
\end{theorem}

\begin{proof}
Fix a small ball $B_r(x^0)\Subset\Omega$, and let $u_r^\star$ be the $p$-harmonic replacement of $u$ in $B_r(x^0)$. For simplicity we assume $x^0=0$ and use Lemma~4.1 in \cite{LeideQTei15} (which holds for $1<p<2$ as well) to have
\begin{align*}
    \int_{B_\rho}|\D u-\mean{\D u}_\rho|^p\le C\left(\frac\rho r\right)^{n+p\al_p}\int_{B_r}|\D u-\mean{\D u}_r|^p+C\int_{B_r}|\D(u-u_r^\star)|^p,\quad 0<\rho<r.
\end{align*}
Here, $\al_p\in(0,1)$ is some constant depending on $p$, and $\mean{\D u}_s=\frac1{|B_s|}\int_{B_s}u$, $0<s\le r$, is the mean value in $B_s$. Suppose we have \begin{align}
    \label{eq:min-p-har-grad-diff-est-1}
    \int_{B_r}|\D(u-u_r^\star)|^p\le Cr^{n+p\al}
\end{align}
for some constant $\al\in(0,\al_p)$, depending only on $a$ and $p$. Then $$
\int_{B_\rho}|\D u-\mean{\D u}_\rho|^p\le C\left(\frac\rho r\right)^{n+p\al_p}\int_{B_r}|\D u-\mean{\D u}_r|^p+Cr^{n+p\al}, \quad0<\rho<r.
$$
By applying Lemma~3.4 in \cite{HanLin97}, we obtain $$
\int_{B_\rho}|\D u-\mean{\D u}_\rho|^p\le C\rho^{n+p\al},
$$
and thus $\D u\in C^{0,\al}$ by Campanato space embedding theorem.

Now we prove \eqref{eq:min-p-har-grad-diff-est-1} and close the argument. To this aim, we recall the first inequality in \eqref{eq:min-p-har-grad-diff-est}: \begin{align}
    \label{eq:Min-p-har-grad-diff-est-2}
    \int_{B_r}|\D(u-u_r^\star)|^p\le C\left(\int_{B_r}(|\D u|^p-|\D u_r^\star|^p)\right)^{p/2}\left(\int_{B_r}|\D u|^p\right)^{1-p/2}.
\end{align}
We also observe that for $\gamma:=1+a\in(0,1)$ \begin{align*}
    \int_{B_r}\left(|\D u|^p-|\D u_r^\star|^p\right)&\le \int_{B_r}(\hG(u_r^\star)-\hG(u))\,dx=\int_{B_r}\int_{u(x)}^{u_r^\star(x)}\hg(s)\,dsdx\\
    &\le\int_{B_r\cap\{u_r^\star>u\}}\int_{u(x)}^{u_r^\star(x)}C_1s^a\,dsdx\le C\int_{B_r\cap\{u_r^\star>u\}}\left((u_r^\star)^\gamma-u^\gamma\right)\\
    &\le C\int_{B_r\cap\{u_r^\star>u\}}|u_r^\star-u|^\gamma\le C\int_{B_r}|u_r^\star-u|^\gamma,
\end{align*}
where we used Lemma~2.5 in \cite{LeideQTei15} in the second to last step. By applying Poincaré inequality together with Hölder's inequality, we further have \begin{align*}
    \int_{B_r}\left(|\D u|^p-|\D u_r^\star|^p\right)\le C|B_r|^{1+\frac\gamma{n}-\frac\gamma{p}}\left(\int_{B_r}|\D(u-u_r^\star)|^p\right)^{\gamma/p}=Cr^{n+\gamma-\frac{n\gamma}p}\left(\int_{B_r}|\D(u-u_r^\star)|^p\right)^{\gamma/p}.
\end{align*}
Combining this with \eqref{eq:Min-p-har-grad-diff-est-2} and writing $A:=\int_{B_r}|\D(u-u_r^\star)|^p$, we infer that $$
A\le C\left(r^{n+\gamma-\frac{n\gamma}p}A^{\gamma/p}\right)^{p/2}\left(\int_{B_r}|\D u|^p\right)^{1-p/2}.
$$
Recall that we have proved in Lemma~\ref{lem:alm-Lip-reg-p} that for any $\e\in(0,1)$ 
$$
\int_{B_r}|\D u|^p\le Cr^{n-p+p(1-\e)}=Cr^{n-p\e},\quad 0<r<r_\e.
$$
Inserting this into the equation above gives $$
A\le C\left(r^{n+\gamma-\frac{n\gamma}p}A^{\gamma/p}\right)^{p/2}r^{(n-p\e)(1-p/2)},
$$
which can be simplified as $$
A\le Cr^{n+p\left(\frac{\gamma-\e(2-p)}{2-\gamma}\right)}.
$$
Now, we take $\e=\frac\gamma{2(2-p)}$, and get $$
A\le Cr^{n+p\al}, \quad \al=\frac{\gamma}{2(2-\gamma)}.
$$
Redefining $\al=\min\left\{\frac\gamma{2(2-\gamma)},\al_p\right\}$ if necessary, we obtain \eqref{eq:min-p-har-grad-diff-est-1} and completes the proof.
\end{proof}

\section*{Declarations}

\noindent {\bf  Data availability statement:} All data needed are contained in the manuscript.

\medskip
\noindent {\bf  Funding and/or Conflicts of interests/Competing interests:} The authors declare that there are no financial, competing or conflict of interests.


\end{document}